\documentclass[12pt]{article}

\usepackage{amssymb}
\usepackage{amsmath}
\usepackage{graphicx}
\usepackage{wrapfig}
\usepackage{authblk}
\usepackage[all]{xy}
\usepackage{bbm}
\usepackage{float}
\usepackage{tikz}
\usetikzlibrary{arrows,calc,shapes,decorations.pathreplacing}
\usepackage{wrapfig}
\usepackage{sidecap}
\usepackage[bookmarks]{hyperref}
\usepackage{caption}

\newtheorem{theorem}{Theorem}[section] 
\newtheorem{lemma}[theorem]{Lemma}

\newtheorem{corollary}[theorem]{Corollary}

\newenvironment{proof}[1][Proof]{\begin{trivlist}
\item[\hskip \labelsep {\bfseries #1}]}{\end{trivlist}}

\newenvironment{remark}[1][Remark]{\begin{trivlist}
\item[\hskip \labelsep {\bfseries #1}]}{\end{trivlist}}

\newcommand{\qed}{\nobreak \ifvmode \relax \else
      \ifdim\lastskip<1.5em \hskip-\lastskip
      \hskip1.5em plus0em minus0.5em \fi \nobreak
      \vrule height0.75em width0.5em depth0.25em\fi}

\numberwithin{equation}{section}
\begin{document}

\title{Equilibrium Arrival Times to a Queue with Order Penalties}
\author{Liron Ravner\footnote{Department of Statistics and the Center for the Study of Rationality, The Hebrew University of Jerusalem, 91905 Jerusalem, Israel. liron.ravner@mail.huji.ac.il}} 

\date{\today}
\maketitle

\begin{abstract}
Suppose customers need to choose when to arrive to a congested queue with some desired service at the end, provided by a single server that operates only during a certain time interval. We study a model  where the customers incur not only congestion (waiting) costs but also penalties for their index of arrival. Arriving before other customers is desirable when the value of service decreases with every admitted customer. This may be the case for example when arriving at a concert or a bus with unmarked seats or going to lunch in a busy cafeteria.	 We provide game theoretic analysis of such queueing systems with a given number of customers, specifically we characterize the arrival process which constitutes a symmetric Nash equilibrium.
\end{abstract}

\section{Introduction}\label{sec_intro}
When customers are faced with the decision of when to arrive to a queueing system with some desired service at the end, the first issue to consider is avoiding congestion. This disutility is typically modelled as a waiting time cost. Such a model was first considered by Glazer and Hassin \cite{GH1983}. They assumed that the number of customers arriving to the queue is a Poisson random variable, which as it turns out makes the analysis easier than the deterministic case, or any other distribution. However, in many queueing scenarios customers may also be interested in being served at an early time. Such an example is driving home from work, where commuters wish to avoid traffic but are not willing to stay at work until midnight in order to achieve this. This type of disutility has been modelled as a tardiness cost that increases the later one is admitted into service. Some recent research has been carried out on this model by Haviv in \cite{H2013} and by Juneja and Shimkin in \cite{JS2012}. The first considered a Poisson number of customers and studied the equilibrium properties when limiting the allowed arrival period. The latter considered a general number of customers and focused on a rigorous characterization of the Nash equilibrium, and the proof of convergence to the fluid limit. Both also presented fluid approximation models which are technically less cumbersome, and often provide insight on the discrete stochastic case. \\

In many queueing scenarios customers are not actually worried about tardiness, but rather about the number of customers who arrived ahead of them. This is the case in a concert or flight with unmarked seats, when there is no actual penalty for tardiness unless other customers have arrived and taken hold of the better seats. This brings us to the focus of this work, which is to present a model where one's cost is not necessarily time based, but rather dependent on the number of prior arrivals. If this is the only disutility assumed in the model, then obviously all customers arrive as early as possible. However, if there are also waiting costs, customers may improve their utility by not arriving in close proximity to others, which leads to a more interesting analysis of their strategic behaviour. In the rest of this section we introduce the model and review some related literature. 

Our analysis commences in section \ref{sec_two_customer} by illustrating an example of a two customer game, and comparing it to the known results for the tardiness model. We show that the support of the symmetric equilibrium arrival distribution is infinite if there is no closing time for the server, as opposed to the finite support obtained in the tardiness model. In particular, the equilibrium distribution is uniform prior to the opening time, and exponential after it. We further show how the equilibrium is adjusted if early birds are not allowed, and when the server has a closing time. All solutions for the two customer game are explicitly derived.  In section \ref{sec_general} we consider a general model with any number of customers. The explicit solution is not tractable for the general model, but we characterize the equilibrium properties and dynamics, accompanied by a numerical technique to compute it. We also provide an analysis of the tail behaviour of the arrival distribution, and prove that the tail behaviour is exponential. Numerical analysis also suggests that the tail of the arrival hazard rate equals exactly to the exponential rate found in the two customer case (regardless of the population size). We also examine how an immediate generalization can be made to a model with both order and tardiness costs. We further provide bounds on the cost incurred by any single customer in equilibrium, in this general setting. In section \ref{sec_poisson} we present the symmetric equilibrium for an a random number of customers, which follows a Poisson distribution. In section \ref{sec_social} we briefly discuss the social optimization problem and explain how the existing literature relates to the model we have presented. Finally, in section \ref{sec_conclusion} we summarize the results and discuss possible extensions and future work.

\begin{remark}
Much of the "essence" of the model is captured in the two-customer analysis of section \ref{sec_two_customer}. Section \ref{sec_general} is a more technical generalization. The main results are stated in Theorems \ref{General_theorem_noT}, \ref{General_theorem_T} and \ref{General_theorem_F}, and numerical examples for the general model are presented in subsection \ref{sec_general_examples}.
\end{remark}

\subsection{Model}\label{sec_model}
Suppose that $N+1$ customers wish to obtain service. We assume that a single server provides the service according to a \textit{First Come First Served} regime, and that service times are independent and exponentially distributed with rate $\mu$. If multiple customers arrive at exactly the same time, then they are admitted in uniformly random order. The customers incur a delay cost of $\alpha$ per unit of time, a tardiness cost of $\beta$ per unit of time until their admittance into service and an index of arrival cost $\gamma$ for every customer that has arrived before them. We denote the closing time by $T>0$, where $T=\infty$ means there is no closing time. For most of this work we assume that $\beta=0$, and analyse the model with only waiting and order costs. Where possible, we also consider $\beta>0$ for the sake of comparison and generalization. \\

In \cite{JS2012}, the equilibrium for a general $N$ customer was characterized under the assumption that customers are limited to arrival distributions $F$ such that: \textit{"For each $F$, the corresponding support can locally (i.e., on any finite interval) be represented as a finite union of closed intervals and points"}. They proved that under this assumption, the equilibrium arrival profile is unique and symmetric. We focus here only on distributions that satisfy this assumption. \\

The symmetric equilibrium mixed strategy is defined by a cdf denoted by $F(t)$ for all $t \in \mathbbm{R}$. We also denote the density function $f(t)=F^{'}(t)$ for all $t$  such that $F(t)$ is continuous and differentiable. We seek a cdf $F$ such that if the other $N$ customers arrive according to $F$, then the last customer is indifferent between arriving at all points of the support of $F$, and does not prefer any point outside of the support. The expected cost of arriving at time $t \in \mathbbm{R}$ is:
\begin{equation}\label{General_cost}
c_F(t)= -\alpha t \mathbbm{1}_{\lbrace t<0 \rbrace} +\frac{\alpha+\beta}{\mu}\mathbbm{E}Q_F(t)+\beta t \mathbbm{1}_{\lbrace t\geq 0 \rbrace}+\gamma\mathbbm{E} A_F(t),
\end{equation}
where $Q_F(t)$ and $A_F(t)$ are the queue size and the arrival process at time $t$, respectively, when $N$ customers are arriving independently according to $F$. The value of the arrival process at time t, $A_F(t)$, is in fact the index of arrival of the last customer to arrive up until time $t$. In the following sections we will simply denote these processes by $Q(t)$ and $A(t)$, although their distribution is always determined by $F$. Note that $\mathbbm{E}A_F(t)=NF(t)$, but the expected queue size depends on both arrivals and departures, and typically does not have an explicit form as a function of $F$.

\begin{remark}
The majority of our analysis assumes that the size of $N$ is common knowledge. It is important to note however, that all results may be generalized to any prior distribution on $N$ in a fairly straightforward manner. The special case of the Poisson distribution has simplifying properties which we shall elaborate on in section \ref{sec_poisson}.
\end{remark}

\subsection{Preliminary analysis of the index cost model}\label{sec_prelim}
Suppose $\beta=0$ and $\alpha,\gamma>0$, i.e. customers only incur waiting and index costs. This special case of the model has several unique equilibrium properties which will be used throughout our analysis. We state these properties in the following two lemmata and their subsequent corollary. \\
The cost function \eqref{General_cost} can now be rewritten:
\begin{equation}\label{General_index_cost}
c(t)= -\alpha t \mathbbm{1}_{\lbrace t<0 \rbrace} +\frac{\alpha}{\mu} \mathbbm{E}Q(t)+\gamma \mathbbm{E}A(t).
\end{equation}

\begin{lemma}\label{Lemma_index_no_tb}
There exists no symmetric equilibrium arrival profile such that for some finite time $t_b$, all customers have arrived with probability one; $F(t_b)=1$. Furthermore, the expected cost in equilibrium is at most $\gamma$.
\end{lemma}
\begin{proof}
We assume that there exists such an equilibrium arrival profile, and show that this leads to a contradiction. Any customer can achieve the cost $\gamma+\epsilon$ for any $\epsilon>0$ by arriving at a very large $t>t_b$. This is because the probability that the server is still busy approaches zero when $t\rightarrow\infty$. Therefore, the expected cost, denoted by $c_e$, in this equilibrium is at most $\gamma$. This cost is constant on all of the support, specifically at time $t_b$:
\begin{equation}
c_e = c(t_b) = \frac{\alpha}{\mu} \mathbbm{E}Q(t)+\gamma > \gamma
\end{equation}
Contradicting the previous argument, that $c_e\leq\gamma$. \qed
\end{proof}

\begin{lemma}\label{Lemma_no_holes}
There can be no holes in a symmetric equilibrium arrival profile. In other words, there exists no time $t$ such that $F(t)>F(t-)$, where $F(t-)=\lim_{s\uparrow t}F(s)$ is the limit from the left of the cdf at point $t$ (the point of upward discontinuity).
\end{lemma}
\begin{proof}
Assume there exists a time $t$ such that $F(t)>F(t-)$. The left limit from the left of the cost function \eqref{General_index_cost} at $t$ is:
\begin{equation*}
c(t-)= -\alpha t \mathbbm{1}_{\lbrace t<0 \rbrace} +\frac{\alpha}{\mu} \mathbbm{E}Q(t-)+\gamma F(t-).
\end{equation*}
The expected queue size can only have upward jumps, i.e. $\mathbbm{E}Q(s-)\leq\mathbbm{E}Q(s)$ for any time $s$ (see for example Lemma 2 in \cite{JS2012}). Therefore, we can conclude that $c(t-)<c(t)$, which contradicts the equilibrium assumption.
\qed
\end{proof}

\begin{corollary}\label{Corollary_index_interval}
The support of the equilibrium distribution $F$ can be represented as an interval $[t_a,\infty)$, for some finite and negative $t_a$.
\end{corollary}

Note that $t_a>-\infty$ because $\lim_{t\rightarrow -\infty}c(t)=\infty$ for any $F$, and in Lemma \ref{Lemma_index_no_tb} we established that the equilibrium cost is at most $\gamma$.

\subsection{Related literature}\label{sec_literature}
In 1969 Vickrey published his seminal paper \textit{"Congestion Theory and Transport Investment"} \cite{V1969}, which presented a fluid model for congestion dynamics. In particular, Nash equilibrium arrival dynamics where characterized for a bottleneck model. This model was studied and developed in various directions in the following decades, among many others by Arnott, de Palma and Lindsey \cite{ADL1993}, Verhoef \cite{V2003} and Otsubo and Rapoport \cite{OR2008}. The latter is the most closely related to this work because they assume a discrete (non-fluid) number of customers. However, they consider a discrete action space, i.e. the arrivals take place at predefined unit intervals. 

Most of the research on the strategic behaviour in queues, deals primarily with models with a discrete action space, such as whether to join a queue or not or choosing between several queues. Furthermore, the analysis is often focused on the steady state equilibrium properties. The book by Hassin and Haviv \cite{HH2003} presents many results in this field and surveys the existing literature. As mentioned above, the first to examine a model where customers choose their arrival time to a queue were Glazer and Hassin in \cite{GH1983}. In \cite{GH1987} they extended this analysis to a server with bulk service, for example a bus that can carry many customers at once. The main technical differences in this kind of analysis are the continuous action space, and the fact that the analysis requires transient analysis, rather than stationary. In recent years there have been several developments in the research of these models. Among them are the above mentioned works by Haviv \cite{H2013} and Juneja and Shinkin \cite{JS2012}. In the fluid setting, models with non-homogeneous customer types were studied by Jain, Juneja and Shimkin in \cite{JJS2011} and Honnappa and Jain in \cite{HJ2010} for a queueing network. An additional fluid model with a random volume of arrivals was analysed in \cite{JRS2012} by Juneja, Raheja and Shimkin. Hassin and Kleiner \cite{HK2010} studied a model with opening and closing times, and no tardiness cost. In \cite{HR2013}, Haviv and Ravner examine a multi-server system with no queue buffer, where customers are interested in maximizing the probability of obtaining service. Honnappa, Jain and Ward characterize the general properties of a single server queue with individual customer arrivals in \cite{HJW2012}, along with finding the transient fluid and diffusion limits of the process. In \cite{HKK2010}, Haviv, Kella and Kerner analyse a model where customers observe their order of arrival, but the decision variable is still binary: join or balk. Our aim is to add to this body of work by introducing a new cost structure which takes into account the order of the arrivals. 

\section{Two customer arrival game}\label{sec_two_customer}
Let $N=1$, which means each customer knows there is only one other customer. We will present explicit equilibrium solutions for several scenarios and compare them to the known result of the tardiness cost model.

\subsection{Tardiness cost}\label{sec_two_tardiness}
Suppose that $\gamma=0$ and $\beta>0$. In \cite{JS2012} the two customer case was explicitly solved. They showed that the symmetric equilibrium arrival distribution $F$ with density $f$ which is uniform before time zero linearly decreasing after:
\begin{equation} \label{JS_equil_f}
f(t) =
\left\{
	\begin{array}{ll}
		\mu \frac{\alpha}{\alpha+\beta} \mbox{, } &  t \in [t_a,0) \\
		-\mu \frac{\beta}{\alpha+\beta} -\frac{\mu^2}{\alpha+\beta}(\beta t+\alpha t_a)\mbox{, } &  t \in [0,t_b]  \\
		0\mbox{, } &  o.w.
	\end{array}
\right. ,
\end{equation}
where
\begin{equation}
-t_a=\frac{1}{\mu} \sqrt{\frac{\beta}{\alpha}\left(2+\frac{\beta}{\alpha}\right)}
\end{equation}
and
\begin{equation}
t_b=\frac{1}{\mu} \left( \sqrt{1+\frac{2\alpha}{\beta}}-1\right)
\end{equation}
define the support of $F$. Note that the density $f$ has a a downwards discontinuity at time zero. Finally, the expected cost for each customer in equilibrium equals $-\alpha t_a$. \\

\begin{figure}[h!]
\centering
\begin{tikzpicture}[xscale=5,yscale=0.85]
  \def\xmin{-0.5}
  \def\xmax{1}
  \def\ymin{0}
  \def\ymax{2.5}
    \draw[->] (\xmin,\ymin) -- (\xmax,\ymin) node[right] {$t$} ;
    \draw[->] (0,\ymin) -- (0,\ymax) node[above] {$f(t)$} ;
    \foreach \x in {-0.5,-0.29,0,0.55,1}
    \node at (\x,\ymin) [below] {\x};
    \foreach \y in {1,2}
    \node at (0,\y) [left] {\y};
    \draw[red, ultra thick, domain=-0.5:-0.294] plot (\x, {0});
	\draw[red, ultra thick, domain=-0.294:0] plot (\x, {2.25});
	\draw[red, ultra thick, domain=0:0.549] plot (\x, {1.2343	-2.25*\x});
	\draw[red, ultra thick, domain=0.54986:1] plot (\x,{0});
	\draw[red,-, densely dashed] (0,1.2343) -- (0, 2.25);
	\draw[red,-, densely dashed] (-0.294,0) -- (-0.294, 2.25);
\end{tikzpicture}
\caption{Equilibrium arrival density - tardiness cost \ $\left(\mu=3, \ \alpha=6, \ \beta=2\right)$} 
\end{figure}
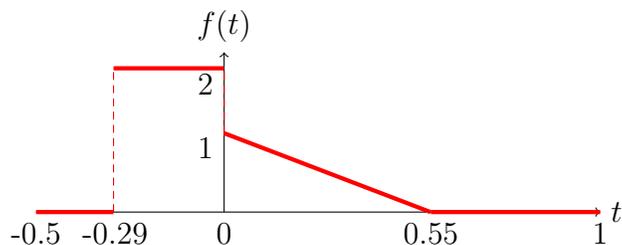

\subsection{Index cost}\label{sec_two_index}
Now we consider the case where the waiting cost is the same as above, but instead of the tardiness cost a customer incurs a cost of $\gamma$ if the other customer has arrived before him. From Corollary \ref{Corollary_index_interval} we know that any equilibrium arrival distribution in this case is given by a continuous cdf $F$ defined on a support interval $[t_a,\infty)$, where $-\infty<t_a<0$. We next show that the equilibrium arrival density is again uniform before time zero, but after time zero it is exponential. 

\begin{theorem}
The following density defines a unique symmetric equilibrium arrival distribution:
\begin{equation} \label{Index_equil_f}
f(t) =
\left\{
	\begin{array}{ll}
		 \frac{\alpha}{\gamma+\frac{\alpha}{\mu}} & \mbox{, } t \in \left[-\frac{\gamma}{\alpha},0 \right] \\
		\frac{\mu}{\left(1+\frac{\alpha}{\gamma \mu}\right)\left(1+\frac{\gamma \mu}{\alpha}\right)} e^{-\frac{\mu}{1+\frac{\alpha}{\gamma \mu}}t} & \mbox{, } t \in (0,\infty )  \\
		0 & \mbox{, } o.w.
	\end{array}
\right.
.
\end{equation}
The expected cost for each customer in this equilibrium is $c_e=\gamma$.
\end{theorem}
\begin{proof}
Suppose, without loss of generality, that the first customer arrives according to some distribution with cdf $G$ and density $g$. The following analysis establishes conditions such that playing $G$ is a best response for the second customer, i.e. that there exists some constant $c_e$ such that this is her expected cost on the support of $G$, and at least $c_e$ outside of the support. We will show that the distribution defined by the density $f$ in $\eqref{Index_equil_f}$ is the unique solution that satisfies these conditions.

No service is given before the opening, so all customers who have arrived are still in the queue: $\mathbbm{E}Q(t)=\mathbbm{E}A(t)=G(t), \ \forall t<0$. Therefore, the cost of arriving before the opening, at some time $t<0$ is:
\begin{equation}
c(t)=-\alpha  t +\frac{\alpha}{\mu} G(t)+\gamma G(t).
\end{equation}
The equilibrium condition for $t<0$ is now: 
\begin{equation}
c_e= -\alpha  t +\frac{\alpha}{\mu} G(t)+\gamma G(t).
\end{equation}
By taking derivative we obtain the first part of \eqref{Index_equil_f}, which means that any equilibrium density $g$ satisfies $g(t)=f(t),\ \forall t<0$.

After time zero, the server commences operation and the system can be presented as a non-homogeneous in time Markov process with the following three states: $\left\lbrace (0,0),(0,1),(1,1)\right\rbrace$, where $(i,j)$ is the state that there are $i$ customers in the system and $j$ have already arrived. We denote the probability of state $(i,j)$ by $p_{ij}:=\mathbbm{P}(Q(t)=i,A(t)=j)$. The dynamics of this process satisfy the following set of differential equations:
\begin{eqnarray}
p^{'}_{0,0}(t) &=& -\frac{g(t)}{1-G(t)}p_{0,0}(t) \label{Two_dynamics1} \\
p^{'}_{0,1}(t) &=& \mu p_{1,1}(t) \label{Two_dynamics2} \\
p^{'}_{1,1}(t) &=& \frac{g(t)}{1-G(t)}p_{0,0}(t)-\mu p_{1,1}(t) \label{Two_dynamics3}.
\end{eqnarray}
Note that the the expected queue size is exactly $p_{1,1}(t)$ in this case, and so the equilibrium condition can be written as follows:
\begin{equation} \label{Index_Condition}
c_e=\frac{\alpha}{\mu}p_{1,1}(t) +\gamma G(t),
\end{equation}
and by taking derivatives:
\begin{equation}
p^{'}_{1,1}(t)=-\frac{\gamma \mu}{\alpha} g(t).
\end{equation}
From \eqref{Two_dynamics3} we also have:
\begin{equation}
p^{'}_{1,1}(t) = \frac{g(t)}{1-G(t)}p_{0,0}(t)-\mu p_{1,1}(t) = g(t)-\mu p_{1,1}(t).
\end{equation}

The second equation is due to the fact that the probability of the state $(0,0)$ is the probability that the other customer has not arrived yet, namely $1-G(t)$. And by combining the last two equations we get:
\begin{equation}
p_{1,1}(t)=g(t) \left(\frac{1}{\mu}+\frac{\gamma}{\alpha}\right).
\end{equation}
We have thus far established that $G$ is the solution to the following differential equation:
\begin{equation} \label{Index_diff_eq}
c_e=\frac{\alpha}{\mu}g(t) \left(\frac{1}{\mu}+\frac{\gamma}{\alpha}\right) +\gamma G(t) \ , \ t\geq 0.
\end{equation}
In order to solve this equation explicitly we need to obtain $c_e$ and $G(0)$. We argue that if $G$ is a solution to \eqref{Index_diff_eq} on a support of $[0,\infty)$ then $c_e=\gamma$. This is because if we take a limit to infinity, then $g(t)$ approaches zero and $G(t)$ approaches one, and so the cost approaches $\gamma$. But for this to be equilibrium the cost throughout the support must equal $\gamma$. This implies that the beginning of the negative interval is $t_a=-\frac{\gamma}{\alpha}$ and that $G(0)=\frac{\gamma}{\gamma+\frac{\alpha}{\mu}}$. We complete the proof by stating that the density function $f$ in the positive part of \eqref{Index_equil_f} is the unique solution to the linear differential equation defined by all of the above:
\begin{equation} \label{Index_diff_eq2}
g(t) \left(\frac{\alpha}{\gamma\mu^2}+\frac{1}{\mu}\right)= (1-G(t)) \ , \ G(0)=\frac{1}{1+\frac{\alpha}{\gamma \mu}}.
\end{equation}
\qed
\end{proof}

\begin{remark}[Remark 1]
Note that also in this case, $f$ has a downward discontinuity at zero. This can be argued by a few algebraic steps showing that the following two inequalities are equivalent:
\begin{equation}
\frac{\alpha}{\gamma+\frac{\alpha}{\mu}}>
\frac{\mu}{\left(1+\frac{\alpha}{\gamma\mu}\right)\left(1+\frac{\gamma \mu}{\alpha}\right)}, 
\end{equation}
and
\begin{equation}
\frac{\alpha}{\gamma \mu}>0.
\end{equation}
\end{remark}

\begin{remark}[Remark 2]
The hazard rate of the equilibrium arrival distribution for $t>0$ is $h(t)=\frac{\mu}{1+\frac{\alpha}{\gamma\mu}}$. This constant will appear again in the numerical analysis part of section \ref{sec_general}, where we will show that although the general equilibrium arrival distribution for any number of customers is not exponential, the hazard rate approaches this very same rate (which is independent of the number of customers).
\end{remark}

\begin{figure}[h!]
\centering
\begin{tikzpicture}[xscale=2.8,yscale=1.45]
  \def\xmin{-1}
  \def\xmax{2.5}
  \def\ymin{0}
  \def\ymax{2.4}
    \draw[->] (\xmin,\ymin) -- (\xmax,\ymin) node[right] {$t$} ;
    \draw[->] (0,\ymin) -- (0,\ymax) node[above] {$f(t)$} ;
    \foreach \x in {-1,-0.33,0,1,2}
    \node at (\x,\ymin) [below] {\x};
    \foreach \y in {1,2}
    \node at (0,\y+0.1) [left] {\y};
    \draw[red, ultra thick, domain=-1:-0.333] plot (\x, {0});
	\draw[red, ultra thick, domain=-0.333:0] plot (\x, {2});
	\draw[red, ultra thick, domain=0:2.5] plot (\x, {0.666*exp(-\x)});
	\draw[red,-, densely dashed] (0,0.666) -- (0, 2);
	\draw[red,-, densely dashed] (-0.333,0) -- (-0.333, 2);
\end{tikzpicture}
\caption{Equilibrium arrival density - Index Cost \ $\left(\mu=3, \ \alpha=6, \ \gamma=1\right)$} 
\end{figure}
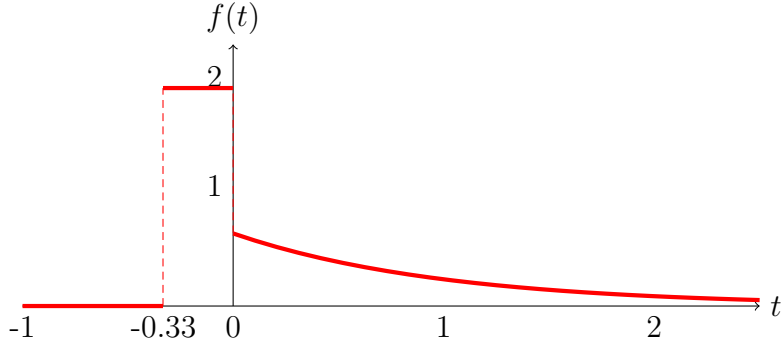

\subsection{Index cost - no early birds}\label{sec_two_index_noearly}
Suppose now that there is no option to queue before the server commences operation. We assume the same index cost function as in the previous subsection. First we argue that in a symmetric equilibrium, customers arrive at time zero with some positive probability, denoted by $p_0>0$. For if this was not the case, then any customer could reduce her own cost to zero by arriving at time zero with probability $1$. Recall that if both customers arrive at the same time then they are admitted into service in a random order. Given the equilibrium probability $p_0$, the cost of arriving at time zero is:
\begin{equation} \label{two_no_early}
c(0)=\frac{p_0}{2}\left(\frac{\alpha}{\mu}+\gamma\right).
\end{equation}
As in the previous subsection, any customer can guarantee a cost as close as she desires to $\gamma$ by arriving late enough. Therefore, if $p_0<1$ then we have $c(0)=\gamma$. This is the outcome when the cost of waiting for the service completion of the other customer is higher than the cost of arriving last: $\frac{\alpha}{\mu}>\gamma$. Otherwise, both customers can ensure a cost that is lower than $\gamma$ (but higher than $\frac{\gamma}{2}$) by arriving at time zero with probability $1$. The conclusion from the above is that in equilibrium the probability to arrive at time zero is:
\begin{equation}\label{two_p0}
p_0 =
\left\{
	\begin{array}{ll}
		 \frac{2\gamma}{\gamma+\frac{\alpha}{\mu}} & \mbox{, } \frac{\alpha}{\mu}>\gamma  \\
		1 & \mbox{, }  \frac{\alpha}{\mu}\leq\gamma
	\end{array}
\right. .
\end{equation}
The corresponding equilibrium costs are $\gamma$ and $\frac{1}{2}\left(\frac{\alpha}{\mu}+\gamma\right)$, respectively. \\
Next we seek the arrival distribution after time zero for the case where $\frac{\alpha}{\mu}>\gamma$. If the other customer has arrived at time zero, then the probability that he is still in service at time $t$ is $e^{-\mu t}$. Therefore, if the arrival probability at zero is $p_0$ and $0$ on the interval $(0,t]$, then the expected cost at time $t$ is:
\begin{equation}\label{two_no_early_empty_cost}
c(t)= \frac{2\gamma}{\gamma+\frac{\alpha}{\mu}}\left(\gamma+e^{-\mu t}\frac{\alpha}{\mu}\right).
\end{equation}
This is a decreasing function of $t$ that initiates at $c(0)=2\gamma$. Therefore, there exists some $t^e$ such that arriving anywhere in $(0,t^e)$ costs more than the equilibrium cost of $\gamma$.  Thus, in equilibrium $f(t)=0$ for all $t\in (0,t^e)$. By comparing \eqref{two_no_early_empty_cost} to $\gamma$ we conclude that $t^e=-\frac{1}{\mu}\log\left(\frac{1}{2}\left(1-\frac{\gamma\mu}{\alpha}\right)\right)$. Finally, we are left with computing $f(t)$ for all $t\geq t^e$. The system dynamics and the equilibrium conditions are identical to the ones presented in the early bird case with the initial condition $F(t^e)=p_0$ \footnote{It makes no difference for the analysis if the customers have all arrived at an atom or uniformly on an interval, as long as the initial conditions are the same.}. Therefore, the density for $t>t^e$ is: $f(t) =	(1-p_0)\frac{\mu}{2} p_0 e^{-\frac{\mu}{2} p_0(t-t_e)}$. The following theorem summarizes the results of this subsection.

\begin{theorem}
For the two customer game with no early birds the symmetric equilibrium arrival distribution is given by:
\begin{itemize}
\item[1.] If $\frac{\alpha}{\mu}\leq\gamma$ then $p_0=1$, i.e. both customers arrive at time zero and are admitted into service in random order. The expected cost for each customer is $\frac{1}{2}\left(\frac{\alpha}{\mu}+\gamma\right)$.
\item[2.] If $\frac{\alpha}{\mu}>\gamma$ then $p_0=\frac{2\gamma}{\gamma+\frac{\alpha}{\mu}}$ and: 
\begin{equation}
f(t)=(1-p_0)\frac{\mu}{2} p_0 e^{-\frac{\mu}{2} p_0(t-t_e)}\mathbbm{1}_{\lbrace t\geq t^e\rbrace},
\end{equation} 
where $t^e=-\frac{1}{\mu}\log\left(\frac{1}{2}\left(1-\frac{\gamma\mu}{\alpha}\right)\right)$. The expected cost for each customer is $\gamma$.
\end{itemize} 
\end{theorem}

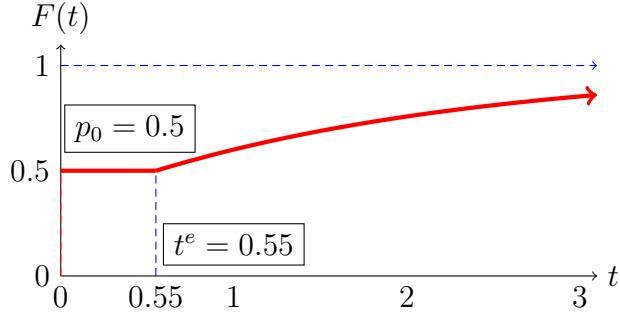
\begin{figure}[h]
\centering
\begin{tikzpicture}[xscale=2.3,yscale=2.8]
  \def\xmin{0}
  \def\xmax{3.1}
  \def\ymin{0}
  \def\ymax{1.1}
    \draw[->] (\xmin,\ymin) -- (\xmax,\ymin) node[right] {$t$} ;
    \draw[->] (0,\ymin) -- (0,\ymax) node[above] {$F(t)$} ;
    \foreach \x in {0,0.55,1,2,3}
    \node at (\x,\ymin) [below] {\x};
    \foreach \y in {0,0.5,1}
    \node at (0,\y) [left] {\y};
    \draw[red, ultra thick, domain=0:0.55] plot (\x, {0.5});
	\draw[red,->, ultra thick, domain=0.55:3.1] plot (\x, {1-0.5*exp(-0.5*(\x-0.55))});
	\draw[blue,-, densely dashed] (0.55,0) -- (0.55,0.5);
	\draw[red,-, densely dashed] (0,0) -- (0,0.5);
	\draw[blue,->, densely dashed] (0,1) -- (3.1,1);
	\node[draw] at (0.4,0.7) {$p_0=0.5$};
	\node[draw] at (1,0.15) {$t^e=0.55$};
\end{tikzpicture}
\caption{Equilibrium arrival \textit{cdf} - index cost and no early birds \ $\left(\mu=2, \ \alpha=6, \ \gamma=1\right)$} 
\end{figure}

\begin{remark}[Remark 1]
If $\frac{\alpha}{\mu}<\gamma$ then $\frac{1}{2}\left(\frac{\alpha}{\mu}+\gamma\right)<\gamma$, which means that the equilibrium cost of both customers decreased when the opportunity to arrive early was taken away. Overall, we have shown that not allowing early birds increases social welfare.
\end{remark}

\begin{remark}[Remark 2]
We overlook the possible existence of non symmetric equilibria of the form: customer $1$ arrives at zero and customer $2$ arrives at some time $T$ such that her cost is lower than $\frac{1}{2}\left(\frac{\alpha}{\mu}+\gamma\right)$. If such an equilibrium exists, then by swapping the customer's roles we get an additional equilibrium, and generally there may be multiple equilibria of this form.
\end{remark}

\subsection{Index cost - finite closing time}\label{sec_two_index_closing}
Suppose again that early birds are allowed, but entry to the system is only possible up until some time $T<\infty$. Service is provided to all who arrived prior to time $T$ even if it is completed after the closing time. The arrival density prior to time zero is as in the no closing time case: $f(t)=\frac{\alpha}{\gamma+\frac{\alpha}{\mu}}$. Customers can no longer ensure a cost of no more than $\gamma$ by arriving very late. Moreover, the equilibrium cost is strictly higher than $\gamma$, because the cost at closing time is $c(T)=\gamma+\mathbbm{P}(Q(t)=1)\frac{\alpha}{\mu}$. Thus, we can conclude that arrivals start at an earlier time than in the previous case: $t_a<-\frac{\gamma}{\alpha}$ (as before, the equilibrium cost is still $-t_a\alpha$). Given $t_a$ we can compute $F(0)= \frac{-t_a \alpha}{\gamma+\frac{\alpha}{\mu}}$, and the dynamics after time zero are as before. From \eqref{Index_diff_eq} we derive the equilibrium equation: 
\begin{equation}\label{Index_diff_eq3}
-t_a \alpha=\frac{\alpha}{\mu}f(t) \left(\frac{1}{\mu}+\frac{\gamma}{\alpha}\right) +\gamma F(t), \ 0\leq t <T.
\end{equation}
Coupled with the initial condition $F(0)$, the solution of this equation is:
\begin{equation}
F(t)=-t_a \frac{\alpha}{\gamma}\left(1-\frac{e^{-\frac{\mu t}{1+\frac{\alpha}{\gamma\mu}}}}{1+\frac{\gamma\mu}{\alpha}}\right).
\end{equation}
By taking $F(T)=1$ we can explicitly derive: $t_a=-\frac{\gamma}{\alpha}\left(1-\frac{e^{-\frac{\mu T}{1+\frac{\alpha}{\gamma \mu}}}}{1+\frac{\gamma\mu}{\alpha}}\right)^{-1}$. Recall that the equilibrium cost is $-t_a\alpha$, thus the cost is a decreasing function of $T$ that approaches $\gamma$, i.e. the cost with no closing time. We can now plug $t_a$ in $F(0)$ and $F(t)$, for $0\leq t\leq T$:
\begin{equation}
F(t)=\frac{\left(1-\frac{e^{-\frac{\mu t}{1+\frac{\alpha}{\gamma \mu}}}}{1+\frac{\gamma\mu}{\alpha}}\right)}{\left(1-\frac{e^{-\frac{\mu T}{1+\frac{\alpha}{\gamma \mu}}}}{1+\frac{\gamma\mu}{\alpha}}\right)}, \ 0\leq t <T \ , \ F(0)=\frac{\left(1-\frac{e^{-\frac{\mu T}{1+\frac{\alpha}{\gamma\mu}}}}{1+\frac{\gamma\mu}{\alpha}}\right)}{1+\frac{\alpha}{\gamma\mu}}.
\end{equation}
We summarize the above analysis in the next theorem.
\begin{theorem}
The equilibrium arrival distribution for the two customer game with closing time $T$ is defined by the following density:
\begin{equation} \label{Index_equil_T_f}
f(t) =
\left\{
	\begin{array}{ll}
		 \frac{\alpha}{\gamma+\frac{\alpha}{\mu}} & \mbox{, } t \in \left[-\frac{\gamma}{\alpha}\left(1-\frac{e^{-\frac{\mu T}{1+\frac{\alpha}{\gamma\mu}}}}{1+\frac{\gamma\mu}{\alpha}}\right)^{-1},0 \right) \\
		\frac{\mu}{\left(1+\frac{\alpha}{\gamma\mu}\right)\left(1+\frac{\gamma\mu}{\alpha}-e^{-\frac{\mu T}{1+\frac{\alpha}{\gamma\mu}}}\right)}e^{-\frac{\mu t}{1+\frac{\alpha}{\gamma\mu}}} & \mbox{, } t \in [	0,T]  \\
		0 & \mbox{, } o.w.
	\end{array}
\right. ,
\end{equation}
and the equilibrium cost of every customer is $\gamma\left(1-\frac{e^{-\frac{\mu T}{1+\frac{\alpha}{\gamma\mu}}}}{1+\frac{\gamma\mu}{\alpha}}\right)^{-1}$.
\end{theorem}

\begin{figure}[h!]
\centering
\begin{tikzpicture}[xscale=4,yscale=1.6]
  \def\xmin{-1}
  \def\xmax{1.2}
  \def\ymin{0}
  \def\ymax{1.8}
    \draw[->] (\xmin,\ymin) -- (\xmax,\ymin) node[right] {$t$} ;
    \draw[->] (0,\ymin) -- (0,\ymax) node[above] {$f(t)$} ;
    \foreach \x in {-1,-0.38,0,1,}
    \node at (\x,\ymin) [below] {\x};
    \foreach \y in {1,1.5}
    \node at (0,\y+0.1) [left] {\y};
    \draw[red, ultra thick, domain=-1:-0.38] plot (\x, {0});
    \draw[red, ultra thick, domain=-0.38:0] plot (\x,{1.5});
	\draw[red, ultra thick, domain=0:1] plot (\x, {0.844*exp(-1.5*	\x)});
	\draw[red, ultra thick, domain=1:1.2] plot (\x, {0});
	\draw[red,-, densely dashed] (-0.38,0) -- (-0.38,1.5);
   	\draw[red,-, densely dashed] (0,0.844) -- (0,1.5);
	\draw[red,-, densely dashed] (1,0) -- (1,0.1884);
\end{tikzpicture}
\caption{Equilibrium arrival density - index cost and a finite closing time \ $\left(\mu=3, \ \alpha=6, \ \gamma=2, \ T=1 \right)$} 
\end{figure}
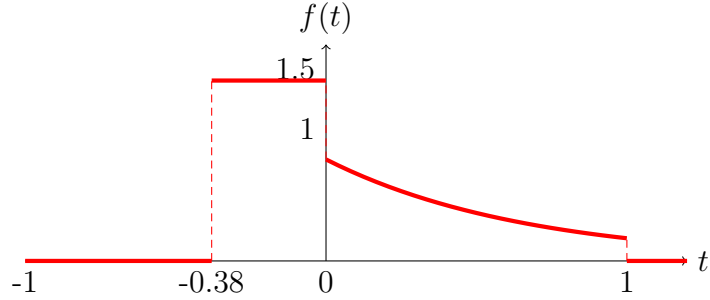

\subsection{Index cost - finite closing time and no early birds}\label{sec_two_index_closing_noearly}
We now wish to combine the two previous scenarios and consider an arrival game where arrivals are only allowed during a finite interval $[0,T]$. We first seek a condition for the strategy of arriving at time zero with probability one to be an equilibrium. This is the case if when the other customer arrives at time zero, the cost of arriving at zero too is lower than arriving at the closing time:
\begin{equation}\label{Index_closed_cost}
\frac{1}{2}\left(\frac{\alpha}{\mu}+\gamma\right)\leq\gamma+\frac{\alpha}{\mu}e^{-\mu T} \Leftrightarrow \frac{\alpha}{\mu}(1-2e^{-\mu T})\leq\gamma .
\end{equation}
Hence, when the parameters satisfy this condition then using the notation of the previous sections we have $p_0=1$. The equilibrium cost here is the expression on the left of \eqref{Index_closed_cost}. Otherwise, we have $p_0<1$ and an equilibrium cost of $\frac{p_0}{2}\left(\gamma+\frac{\alpha}{\mu}\right)$. As before, there is a gap after time zero in the support of the distribution, because for any $p_0$ the cost of arriving at any time $t$ in such a gap is:
\begin{equation}
c(t)=p_0\left(\gamma+\frac{\alpha}{\mu}e^{-\mu t}\right).
\end{equation}
Denote by $t^e$ the first time that this cost equals the equilibrium cost, and we can compute this value by comparing to the equilibrium cost:
\begin{equation}
p_0\left(\gamma+\frac{\alpha}{\mu}e^{-\mu t^e}\right)=\frac{p_0}{2}\left(\gamma+\frac{\alpha}{\mu}\right).
\end{equation}
Clearly, the solution is the same as in the case without a closing time: $t^e=-\frac{1}{\mu}\log\left(\frac{1}{2}\left(1-\frac{\gamma\mu}{\alpha}\right)\right)$. The dynamics after time $t_e$ are as in the previous section with $F(t^e)=p_0$. Thus, we can compute $F(t)$ by solving the equation:
\begin{equation}\label{Index_diff_eq4}
\frac{p_0}{2}\left(\gamma+\frac{\alpha}{\mu}\right)=\frac{\alpha}{\mu}f(t) \left(\frac{1}{\mu}+\frac{\gamma}{\alpha}\right) +\gamma F(t) \ , \ t\in[t^e,T] .
\end{equation}
Finally, together with the equation $F(T)=1$ we obtain:
\begin{equation}
p_0=\frac{2}{1+\frac{\alpha}{\gamma\mu}-\left(\frac{\alpha}{\gamma\mu}-1\right)e^{-\frac{\mu}{1+\frac{\alpha}{\gamma\mu}}(T-t^e)}}
\end{equation} 
and
\begin{equation}
F(t)=\frac{1+\frac{\alpha}{\gamma\mu}-\left(\frac{\alpha}{\gamma\mu}-1\right)e^{-\frac{\mu}{1+\frac{\alpha}{\gamma\mu}}(t-t^e)}}{1+\frac{\alpha}{\gamma\mu}-\left(\frac{\alpha}{\gamma\mu}-1\right)e^{-\frac{\mu}{1+\frac{\alpha}{\gamma\mu}}(T-t^e)}} \ , \ t\in [t^e,T] .
\end{equation}
We conclude this section by stating the resulting theorem:
\begin{theorem}
The equilibrium arrival distribution for the two customer game with arrivals allowed only in the interval $[0,T]$, is given by:
\begin{itemize}
\item[(1)] If $\frac{\alpha}{\mu}(1-2e^{-\mu T})\leq\gamma$ then $p_0=1$, i.e. both customers arrive at time zero and are admitted into service in random order. The expected cost for each customer is $\frac{1}{2}\left(\frac{\alpha}{\mu}+\gamma\right)$.
\item[(2)] If $\frac{\alpha}{\mu}(1-2e^{-\mu T})>\gamma$ then $p_0=\frac{2}{1+\frac{\alpha}{\gamma\mu}-\left(\frac{\alpha}{\gamma\mu}-1\right)e^{-\frac{\mu}{1+\frac{\alpha}{\gamma\mu}}(T-t^e)}}$ and: 
\begin{equation}
f(t)=\frac{p_0\mu}{2}\frac{\frac{\alpha}{\gamma\mu}-1}{\frac{\alpha}{\gamma\mu}+1}e^{-\frac{\mu}{1+\frac{\alpha}{\gamma\mu}}(t-t^e)}\mathbbm{1}_{\lbrace t\in [t^e,T]\rbrace}.
\end{equation}
\item[] The expected cost is $ \frac{p_0}{2}\left(\frac{\alpha}{\mu}+\gamma\right)$ for each customer.
\end{itemize} 
\end{theorem}

\begin{figure}[h!]
\centering
\begin{tikzpicture}[xscale=7,yscale=3.5]
  \def\xmin{0}
  \def\xmax{1.2}
  \def\ymin{0}
  \def\ymax{1.1}
    \draw[->] (\xmin,\ymin) -- (\xmax,\ymin) node[right] {$t$} ;
    \draw[->] (0,\ymin) -- (0,\ymax) node[above] {$F(t)$} ;
    \foreach \x in {0,0.55,1}
    \node at (\x,\ymin) [below] {\x};
    \foreach \y in {0,0.83,1}
    \node at (0,\y) [left] {\y};
    \draw[red, ultra thick, domain=0:0.55] plot (\x, {0.8323});
	\draw[red, ultra thick, domain=0.55:1] plot (\x, {(4-2*exp(-0.5*(\x-0.55)))*0.41615});
	\draw[red,->, ultra thick, domain=1:1.2] plot (\x, {1});
	\draw[red,-, densely dashed] (0,0) -- (0,0.8323);
	\draw[blue,-, densely dashed] (0.55,0) -- (0.55,0.8323);
	\draw[blue,-, densely dashed] (0,1) -- (1,1);
	\draw[blue,-, densely dashed] (1,0) -- (1,1);
	\node[draw] at (0.2,0.7) {$p_0=0.83$};
	\node[draw] at (0.7,0.095) {$t^e=0.55$};
	\node[draw] at (1.12,0.095) {$T=1$};
\end{tikzpicture}
\caption{Equilibrium arrival \textit{cdf} -  index cost, finite closing time and no early birds \ $\left(\mu=2, \ \alpha=6, \ \gamma=1, \ T=1 \right)$} 
\end{figure}
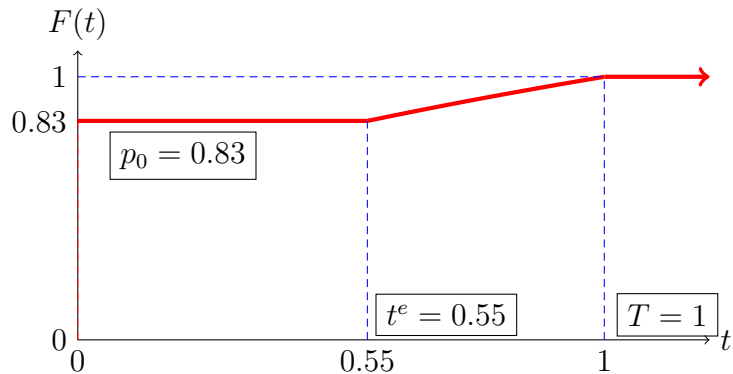

\begin{corollary}
If $\frac{\alpha}{\mu}<\gamma$ then the equilibrium cost is lower than $\gamma$. Otherwise, the equilibrium cost is lower than $\gamma$ if $T\leq\frac{log(2)-log\left(1-\frac{\mu\gamma}{\alpha}\right)}{\mu}$, and at least $\gamma$ if $T>\frac{log(2)-log\left(1-\frac{\mu\gamma}{\alpha}\right)}{\mu}$.
\end{corollary}
The conclusion from the above corollary is that limiting arrivals to an interval $[0,T]$ may increase or decrease the social welfare, dependent on the parameters. And in case $\frac{\alpha}{\mu}\geq\gamma$ a lower cost can be achieved by setting an early enough closing time. This is an interesting observation in light of the previous case, without early birds, where enforcing a closing time always leads to a higher cost.

\section{More than two customers} \label{sec_general}
Suppose now that $N+1>2$ customers require service, and that there are no tardiness costs. In this section we characterize the symmetric equilibrium arrival distribution and discuss its properties. As it turns out, many of the observations made for the two customer game are true for any population size. In particular, we remind the reader that according to Lemma \ref{Lemma_no_holes} there can be no "holes" in the distribution, except for the case with no early birds, where there is one such "hole" that we have fully described. In Theorems \ref{General_theorem_noT} and \ref{General_theorem_T} we summarize the equilibrium properties for $T=\infty$ and $T<\infty$, respectively, for both the case where early birds are allowed and for the case where they are not. In Theorem \ref{General_theorem_F} we state the symmetric equilibrium cdf, explicitly for $t<0$ and as a functional differential equation for $t\geq 0$. This is followed by a generalization to a model with both tardiness and order costs. We will also provide methods for the computation of the equilibria and present some numerical examples. \\

Suppose the server has no closing time and customers are allowed to arrive as late as they wish. In the following Theorem we present the properties of the equilibrium under this assumption. Specifically, we show that the support of the equilibrium arrival distribution is either infinite or a single atom at time zero, which should not come as surprise in light of the two customer game analysis. We also show that the customer's individual equilibrium cost is $N\gamma$ in the case where early birds are allowed, and at most $N\gamma$  if they are not allowed.
\begin{theorem}\label{General_theorem_noT}
If $T=\infty$, then a symmetric equilibrium strategy profile satisfies the following properties. \\ \\
If arriving before time zero is allowed, then:
\begin{itemize}
\item[(1)] The support of the equilibrium arrival distribution is an infinite interval $[t_a,\infty)$, where $t_a<0$.
\item[(2)] The expected cost for all customers in equilibrium is $N\gamma$.
\item[(3)] First possible arrival is at $t_a=-N\frac{\gamma}{\alpha}$. 
\end{itemize}
If arriving before time zero is not allowed, then:
\begin{itemize}
\item[(4)]  there is a positive probability to arrive at time zero, given by:
\begin{equation}\label{two_p0}
p_0 =
\left\{
	\begin{array}{ll}
		 \frac{\gamma}{\gamma+\frac{\alpha}{\mu}} & \mbox{, } \frac{\alpha}{\mu}>\gamma  \\
		1 & \mbox{, }  \frac{\alpha}{\mu}\leq\gamma
	\end{array}
\right. .
\end{equation}
\item[(5)] The equilibrium cost is tightly upper bounded by $N\gamma$.
\item[(6)] If $p_0<1$, the remaining support is defined on an interval $[t_e,\infty)$, where $t_e>0$.
\end{itemize}
\end{theorem}
\begin{proof}
The proof of properties (1) and (2) relies on the same continuity argument used in Lemma \ref{Lemma_index_no_tb}: any customer can ensure a cost almost as low as $N\gamma$ by arriving late enough, and this is lower than the equilibrium cost for any finite support. Property (3) is a direct result of property (2): at time $t_a$ the cost is $-t_a\alpha$, but this must also equal the equilibrium cost along the support, hence we obtain: $t_a=-N\frac{\gamma}{\alpha}$. Property (4) is derived by considering the cost of arriving at time zero if all other customers arrive at time zero with probability one:
\begin{equation}\label{General_zero_cost}
\frac{N}{2}\left(\frac{\alpha}{\mu}+\gamma\right).
\end{equation}
The $(N+1)$'th customer also arrives at time zero if this cost is lower than the cost of arriving at any time $t>0$, and this is indeed the case when the above cost is at most $N\gamma$. Therefore, we deduce that if $\frac{\alpha}{\mu}\leq\gamma$ then $p_0=1$. The cost in this case is exactly \eqref{General_zero_cost} which is clearly smaller than or equal to $N\gamma$. If $\frac{\alpha}{\mu}>\gamma$, then in equilibrium customers arrive at time zero with probability $p_0$ that satisfies: $\frac{Np_0}{2}\left(\gamma+\frac{\alpha}{\mu}\right)=N\gamma$. This yields: $p_0=\frac{\gamma}{\gamma+\frac{\alpha}{\mu}}$. For any $t>0$ we denote by $q_{p}(t)$ the expected queue size at time $t$, if all other customers arrive with probability $p$ at time zero and with probability zero along the interval $(0,t]$. The cost of arriving at time $t>0$ is then:
\begin{equation}
Np_0\gamma+\frac{\alpha}{\mu}q_{p_0}(t) \underset{t\downarrow 0}{\rightarrow}  Np_0\left(\gamma+\frac{\alpha}{\mu}\right).
\end{equation}
Note that this is a monotone decreasing function in $t$. Thus, we can conclude that in equilibrium there is an interval $(0,t_e)$ with no arrivals. The length of this interval $t_e$, satisfies uniquely:
\begin{equation}\label{General_te_equation}
\frac{Np_0}{2}\left(\gamma+\frac{\alpha}{\mu}\right)=Np_0\gamma+\frac{\alpha}{\mu}q_{p_0}(t_e).
\end{equation}
We point out that if $p_0=1$, then the equilibrium cost is lower than $N\gamma$ and if $p_0<1$, then it is exactly $N\gamma$. \qed
\end{proof}

\begin{remark}
The value of $q_p(t)$ which was defined in the above proof can be computed using the properties of the Poisson distribution:
\begin{equation}
q_p(t)=\sum_{j=0}^{N} {N\choose j} p^j(1-p)^{N-j}  \sum_{i=0}^{j}(j-i)\frac{e^{-\mu t}(\mu t)^i}{i!}.
\end{equation}
To complete this definition we define the limit $q_p(\infty):=\lim_{t\rightarrow \infty}q_p(t)=0$. We shall make use of this notation throughout the remainder of this section.
\end{remark}

We now impose a finite closing time for the server, after which customers can no longer arrive. Customers who have arrived before this time will be served even if they are still in service or in the queue at closing time. The next Theorem presents the equilibrium properties of this scenario. Clearly the arrival distribution has a finite support in this case, and the individual equilibrium cost is higher than $N\gamma$.
\begin{theorem}\label{General_theorem_T}
If $T<\infty$, then a symmetric equilibrium strategy profile satisfies the following properties. \\ \\
If arriving before time zero is allowed, then:
\begin{itemize}
\item[(1)] The support of the equilibrium arrival distribution is a finite interval $[t_a,T]$, where $t_a<-N\frac{\gamma}{\alpha}$.
\item[(2)] The expected cost for all customers in equilibrium is larger than $N\gamma$.
\end{itemize}
If arriving before time zero is not allowed, then:
\begin{itemize}
\item[(3)] There is a positive probability to arrive at time zero. Denote this probability by $p_0$. If the following condition is met then $p_0=1$, otherwise $p_0<1$:
\begin{equation}
\frac{N}{2}\left(1-\frac{\gamma\mu}{\alpha}\right)\leq q_1(T).
\end{equation}
\item[(4)] The condition $\frac{\alpha}{\mu}\leq\gamma$ is sufficient (but not necessary) for $p_0=1$.  
\item[(5)] If $p_0<1$, then the remaining support is defined on an interval $[t_e,T]$, where $t_e\in(0,T)$.
\end{itemize}
\end{theorem}
\begin{proof}
We first consider the case where arrivals are allowed at any time $t\leq T$. Suppose the support of $F$ is $[t_a,t_b]$, such that $t_a\leq 0$ and $t_b\leq T$. The cost of arriving at time $t_b$ is $N\gamma+\frac{\alpha}{\mu}\mathbbm{E}Q(t_b)$, and this is also the equilibrium cost. Now assume that $t_b<T$ then the cost of arriving at any time $t\in (t_b,T]$ is lower than the equilibrium cost and this is obviously a contradiction. The above yields properties (1) and (2). \\
Next we consider the case where arrivals are only allowed in the interval $[0,T]$. If all arrive at time zero with probability one, then $p_0=1$ is the best response for the $(N+1)$'th customer if the cost at zero is not be higher than the cost of arriving at $T$:
\begin{equation*}
\frac{N}{2}\left(\frac{\alpha}{\mu}+\gamma\right)\leq N\gamma+\frac{\alpha}{\mu}q_1(t).
\end{equation*} \footnote{If this is an equality, then there are $N+1$ pure and non-symmetric equilibriums of the form: one customer arrives at time $T$ and all others at time zero.}
Simple algebra leads to property (3). Note that if $\frac{\alpha}{\mu}\leq\gamma$, then this condition is met trivially for any $T>0$, hence properties (3) and (4) are obtained. Property (5) here follows the same argument leading to property (6) in the previous theorem. Note that $p_0$ such that $t_e=T$ cannot be an equilibrium because customers would have to arrive at $T$ with probability $1-p_0$, leading to a higher cost than that incurred by arriving at time zero (a strictly positive mass will be added to the cost). The value of $t_e$ is given by equation \eqref{General_te_equation}. \qed
\end{proof}

We have thus far characterized the equilibrium arrival distribution up until time zero. In Lemma \ref{Lemma_no_holes} we showed that the distribution can have no holes or atoms in the interior of the support. Therefore, we seek a continuous distribution defined by a density function $f(t)$ such that $\int_{t_e}^{T}f(t)dt=1-F(0)$. The equilibrium condition is that the cost is constant for all $t$ such that $f(t)>0$. To achieve this we now elaborate on the second relation\footnote{The first relation is the equilibrium condition of a constant cost on all of the support.} between the arrival distribution and the queueing process. The process $\lbrace Q(t):t\geq 0\rbrace$ itself is not Markovian because the number of arrivals after time $t$ is not independent of the number of arrival up to it. However, $\lbrace \left(Q(t),A(t)\right):t\geq 0\rbrace$ is a non-homogeneous in time, two dimensional Markov process satisfying:
\begin{equation}\label{General_Q_dynamics}
\begin{aligned}
	p_{0,j}^{'}(t) &= \mu p_{1,j}(t)-(N-j)h(t)p_{0,j}(t), \ 0\leq j\leq N  \\
	p_{i,j}^{'}(t) &= \mu p_{i+1,j}(t)+(N-j+1)h(t)p_{i-1,j-1}(t)    \\
				   & -(\mu+(N-j)h(t))p_{i,j}(t), \ 1\leq i\leq j\leq N
\end{aligned},
\end{equation}

where $p_{i,j}(t)=\mathbbm{P}(Q(t)=i,A(t)=j)$ for $0\leq i\leq j\leq N$, and $p_{i,j}(t)=0$ otherwise. We denote the hazard rate of $F$ by: $h(t):=\frac{f(t)}{1-F(t)}$. If early birds are allowed, then the initial conditions for these equations are given by:
\begin{equation}\label{General_init_conditions}
p_{i,j}(0)={N \choose i} F(0)^i (1-F(0))^{N-i} \mathbbm{1}_{\lbrace i=j\rbrace}.
\end{equation}
If early birds are not allowed, then these conditions can be modified to:
\begin{equation}
p_{i,j}(t_e)={N \choose j} p_0^j (1-p_0)^{N-j} \frac{e^{-\mu t_e}(\mu t_e)^{j-i}}{(j-i)!} \mathbbm{1}_{\lbrace j\geq i\rbrace}.
\end{equation}
Clearly, $\mathbbm{P}(Q(t)=0)=\sum_{j=0}^N p_{0,j}(t)$. We have already shown that if there is no closing time, then the support of the equilibrium distribution is infinite. In this case, if early birds are allowed then $t_a=-\frac{N\gamma}{\alpha}$ and $F(0)=-t_a\frac{\alpha\mu}{N(\alpha+\gamma\mu)}$. Otherwise, $F(0)=p_0$ which was derived in Theorem \ref{General_theorem_noT}. If there is a closing time, then $t_a$ or $p_0$ need to be determined according to the conditions described in Theorem \ref{General_theorem_T}.

The following Theorem presents the equilibrium arrival distribution using the above defined process.

\begin{theorem}\label{General_theorem_F}
If early birds are allowed, then the equilibrium arrival distribution before time zero is uniform with density:
\begin{equation}\label{General_f_negative}
f(t)=\frac{\alpha\mu}{N(\alpha+\gamma\mu)} \ , \ t\in [t_a,0].
\end{equation}
In both cases, after time zero the equilibrium arrival distribution is characterized by the following functional differential equation:
\begin{equation}\label{General_eq_dynamics}
f(t)=\frac{\alpha\left(1-\mathbbm{P}(Q(t)=0)\right)}{N(\frac{\alpha}{\mu}+\gamma)} \ , \ t\in[t_e,T),
\end{equation}
where $t_e=0$ if early birds are allowed.
\end{theorem}
\begin{proof}
At any time $t<0$ the number of arrivals equals the queue size. Therefore, given that all of the other $N$ customers arrive according to $F$, the cost of arriving at $t<0$ for the $(N+1)$'th customer is:
\begin{equation}\label{General_cost_early}
c(t)= -\alpha t+\left(\frac{\alpha}{\mu}+\gamma\right) NF(t).
\end{equation}
By taking derivative of \eqref{General_cost_early} we immediately obtain \eqref{General_f_negative}.

Utilizing standard queueing dynamics we can state the cost function in the following form:
\begin{equation}\label{General_cost_noTardiness}
c(t)= \frac{\alpha}{\mu}\left(NF(t)-\mu t+\mu\int_{0}^{t} \mathbbm{P} (Q(s)=0)ds\right)+\gamma NF(t),
\end{equation}
and by taking derivative we get \eqref{General_eq_dynamics}.\qed
\end{proof}

\subsection{Tail behaviour of the arrival distribution}\label{sec_general_tail}
In section \ref{sec_two_customer} we showed that the equilibrium arrival distribution is exponential in the two customer game. For a general number of customers this is not the case, but we can show that the arrival distribution is still light tailed. Specifically, that there exists some $\eta>0$ such that:
\begin{equation}\label{eq_light_tail}
\lim_{t\rightarrow\infty}e^{\eta t}(1-F(t))<\infty.
\end{equation}
We further show that there exists some $\eta<\mu$ such that $e^{\eta t}(1-F(t))$ has a finite and non-zero limit. Numerical analysis suggests $\eta=F(0)\mu$, but we were unable to explicitly prove this limit. We will provide an outline of how more accurate analysis of the tail behaviour can be conducted in the future.

We first prove a supporting lemma that states that if the tail of the hazard rate is increasing, then the distribution is light tailed. We then proceed to show that the equilibrium arrival distribution characterized in Theorem \ref{General_theorem_F} satisfies this condition.

\begin{lemma}\label{General_hazard_tail}
Let $X$ be a non-negative and continuous random variable with \textit{cdf} $F$ and density $f$. If there exists a $\tau$ such that for all $t>\tau$ the hazard rate is non-decreasing: $h(t+s)-h(t)\geq 0, \ \forall s>0$, then $X$ is light tailed in the sense of \eqref{eq_light_tail}.
\end{lemma}
\begin{proof}
Denote the tail probability of $X$ by $\overline{F}(t)=1-F(t)$. Recall that an equivalent definition of the hazard function for continuous random variables is $H(T)=\int_{0}^t h(u)du=-log\overline{F}(t)$. So the tail probability can be represented using the hazard rate:
\begin{equation}\label{eq_tail_hazard}
\overline{F}(t)=e^{-\int_{0}^{t}h(u)du}.
\end{equation}
If there exists some $\tau>0$ after which the hazard rate is non-decreasing, then for any $t>\tau$:
\begin{equation*}
\int_{0}^{t}h(u)du\geq K_\tau+(t-\tau)h(\tau),
\end{equation*}
where $K_\tau=\int_{0}^{\tau}h(u)du$. Thus, we can bound the tail probability:
\begin{equation*}
\overline{F}(t)\leq e^{-K_\tau-(t-\tau)h(\tau)}=e^{- h(\tau)t}e^{Th(\tau)-K_\tau}.
\end{equation*}
Finally, if we denote $\eta=h(\tau)$ then we can conclude that the distribution is indeed light tailed:
\begin{equation*}
\lim_{t\rightarrow\infty}e^{\eta t}\overline{F}(t)\leq e^{\eta\tau-K_\tau}<\infty.
\end{equation*}
\qed
\end{proof}

\begin{lemma}\label{General_lemma_exponential_tail}
The equilibrium arrival distribution characterized in Theorem \ref{General_theorem_F} satisfies the condition of Lemma \ref{General_hazard_tail}, and is therefore light tailed. Moreover, there exists some $\eta<\mu$ such that:
\begin{equation}
\lim_{t\rightarrow\infty}e^{\eta t}(1-F(t))=C,
\end{equation}
where $0<C<\infty$.
\end{lemma}
\begin{proof}
We prove this lemma by using the properties of the equilibrium and the underlying queueing process to show that the tail of the hazard function is bounded and non-decreasing, and therefore both satisfies the condition of Lemma \ref{General_hazard_tail} and has a non-zero limit. 

If we denote $p_i(t):=\mathbbm{P}(Q(t)=i)$, then according to Theorem \ref{General_theorem_F} the density for $t>0$ is $f(t)=\frac{1-p_0(t)}{N\left(\frac{1}{\mu}+\frac{\gamma}{\alpha}\right)}$, hence the hazard rate is:
\begin{equation}\label{eq_hazard}
h(t)=\frac{f(t)}{\int_t^\infty f(u)du}=\frac{1-p_0(t)}{\int_t^\infty 1-p_0(u)du}.
\end{equation}
Consider the dynamics of the process $\{Q(t),A(t)\}$ as defined in \eqref{General_Q_dynamics}. By taking a sum on the number of arrivals we have:
\begin{equation}\label{eq_p0_deriv}
\begin{split}
p_0^{'}(t) &= \sum_{j=0}^N p_{0,j}(t)=\mu\sum_{j=0}^N p_{1,j}(t)-\sum_{j=0}^N(N-j)h(t)p_{0,j}(t)\\
&=\mu p_1(t)-h(t)\sum_{j=0}^{N-1}(N-j)p_{0,j}(t)
\end{split}.
\end{equation}
The state $(0,N)$ is clearly an absorbing one, as after all customers have arrived and have been served, there will be no more arrivals or departures. This implies the following when $t$ goes to infinity:
\begin{enumerate}
\item $p_{0,N}(t)\rightarrow 1$ and $p_{i,j}(t)\rightarrow 0, \ \forall (i,j)\neq(0,N)$
\item $p_{0}(t)\rightarrow 1$ and $p_i(t)\rightarrow 0, \ \forall i\geq 1$
\item $p_{0}^{'}(t)\rightarrow 0$.
\end{enumerate}
We characterize the tail behaviour, i.e. for large values of $t$, of the hazard rate by applying L'Hopital's Rule:
\begin{equation}\label{eq_hazard_tail_approx}
h(t) \sim \frac{p_0^{'}(t)}{1-p_0(t)}=\frac{\mu p_1(t)-h(t)\sum_{j=0}^{N-1}(N-j)p_{0,j}(t)}{\sum_{i=1}^Np_i(t)}.
\end{equation}
We first note that the \textit{RHS} is upper bounded by $\mu$:
\begin{equation}\label{eq_hazard_bound}
h(t) < \frac{\mu p_1(t)}{p_1(t) } = \mu.
\end{equation}
We further show that the tail of the hazard rate is strictly increasing. We define the function:
\begin{equation}\label{eq_hazard_g}
g_x(t):=\frac{F(t+x)-F(t)}{1-F(t)}.
\end{equation}
If $g_x(t)$ is increasing in $t$, for any $x>0$, then the hazard rate is increasing too (this is easily verified by \eqref{eq_tail_hazard}). We observe the tail behaviour of $g_x(t)$:
\begin{equation}
g_x(t) = \frac{\int_t^{t+x}1-p_0(u)du}{\int_t^{\infty}1-p_0(u)du}=\frac{\int_t^{\infty}1-p_0(u)du-\int_{t+x}^{\infty}1-p_0(u)du}{\int_t^{\infty}1-p_0(u)du}.
\end{equation}
The denominator is decreasing to zero at rate $1-p_0(t)$, while the numerator is decreasing at a slower rate of $1-p_0(t)-(1-p_0(t+x))$. Therefore the function $g_x(t)$ is increasing w.r.t. $t$, and so is $h(t)$.
\qed
\end{proof}

More accurate analysis of the tail behaviour may perhaps be achieved by applying the tools of Quasi-Stationary finite state Markov chains (see \cite{vDP2013}). Consider the approximation of the tail given in \eqref{eq_hazard_tail_approx}: 
\begin{equation*}
\begin{split}
h(t) &\sim \mu\frac{p_1(t)}{1-p_0(t)}-\frac{h(t)}{1-p_0(t)}\sum_{j=0}^{N-1}(N-j)p_{0,j}(t) \\
&\sim \mu\frac{p_1(t)}{1-p_0(t)}-p_0^{'}(t)\sum_{j=0}^{N-1}(N-j)p_{0,j}(t) \\
&\sim \mu\frac{p_1(t)}{1-p_0(t)}
\end{split}.
\end{equation*}
The last approximation is given by the fact that both terms in the negative product go to zero, by the properties of the process as described above. We are left with showing that $\frac{p_1(t)}{1-p_0(t)}$ has a non-zero limit when $t$ goes to infinity. This is in fact a conditional distribution of being in state $\{Q(t)=1\}$ given that the queue is not empty. The limit of this term can then be seen, after slight modification, as a quasi-stationary distribution \cite{vDP2013}. In other words, as the limit probability of being in a transient state conditioned on the fact that the absorbing state has not been reached. Another point of interest in this context is characterizing the time until absorption, which in our setting is the time until all customers have arrived and have been served. The reason that the standard results which are surveyed by van Doorn and Pollett in \cite{vDP2013} cannot be applied here, is that the process at hand is not homogeneous in time. There are methods for dealing with stationary distributions of non-homogeneous Markov chains, such as the ones presented by Abramov and Lipster in \cite{AL2004}, which may yield more accurate analysis of the tail behaviour of the process of interest in this work. We leave these avenues for future research.

\subsection{Tardiness and index costs}\label{sec_general_tardiness}
Suppose now that customers may incur both tardiness and index costs; $\beta\geq 0$ and $\gamma\geq 0$. If early birds are allowed and there is no closing time, then analysis similar to the above yields:
\begin{theorem}\label{General_theorem_tardiness_F}
The equilibrium arrival distribution is given by:
\begin{equation}\label{General_f_negative_tardiness}
f(t)=\frac{\alpha\mu}{N(\alpha+\beta+\gamma\mu)} \ , \ t\in [t_a,0)
\end{equation}
and:
\begin{equation}\label{General_eq_dynamics_tardiness}
f(t)=\frac{\mu(\alpha+\beta)\left(1-\mathbbm{P}(Q(t)=0)\right)-\beta\mu}{N(\alpha+\beta+\gamma\mu)} \ , \ t\in[0,t_b],
\end{equation}
where $t_b<\infty$ is the upper bound of the arrival support. 
\end{theorem}

The positive part of the density in Theorem \ref{General_theorem_tardiness_F} is equivalent to:
\begin{equation}
f(t)\frac{N}{\mu}=\frac{\alpha}{\alpha+\beta+\gamma\mu}-\frac{\alpha+\beta}{\alpha+\beta+\gamma\mu}\mathbbm{P}(Q(t)=0) \ , \ t\in[0,t_b],
\end{equation} 
and in case $\gamma=0$ we get:
\begin{equation}
f(t)\frac{N}{\mu}=\frac{\alpha}{\alpha+\beta}-\mathbbm{P}(Q(t)=0) \ , \ t\in[0,t_b],
\end{equation}
which coincides with the result obtained in \cite{JS2012}. 

In the following lemma we provide equilibrium conditions for $t_b$ and argue that the support of the equilibrium distribution is indeed finite.
\begin{lemma}\label{General_lemma_tb}
If $\beta>0$ then there exists some finite time $t_b>0$ such that $F(t_b)=1$, $f(t_b)=0$ and $F(t)<1 \ , \ \forall t<t_b$.
\end{lemma}
\begin{proof}
Let us denote the equilibrium expected cost by $c_e$. This cost is clearly finite since $p_0\leq 1$ and $|t_a|<\infty$ (the latter stems from the fact that $F(0)<1$ \footnote{If $F(0)=1$ then there would be a downward discontinuity of the cost at time zero, contradicting the equilibrium assumption.}). If $\beta>0$, then the cost function $c(t)$ is clearly unbounded because the term $\beta t$ is not bounded and all the other terms are positive. Therefore, there must exist some time $t_b$ such that $F(t_b)=1$ and $F(t)<1 \ , \ \forall t<t_b$. To complete the proof we assume that $f(t_b)>0$ and verify that this leads to a contradiction of the equilibrium assumption. Consider the cost at any time $t>t_b$:
\begin{equation} \label{General_cost_tb}
c(t)=N\gamma+\frac{\alpha +\beta}{\mu}\mathbbm{E}Q(t)+\beta t.
\end{equation}
Recall also that for any $t\leq t_b$:
\begin{equation}\label{General_cost_derivative}
c'(t)=f(t)\left(\frac{\alpha+\beta}{\mu}+\gamma\right)-(\alpha+\beta)(1-\mathbbm{P}(Q(t)=0))+\beta,
\end{equation}
and by taking derivative of \eqref{General_cost_tb} for $t>t_b$ we have:
\begin{equation}
c'(t)=\frac{\alpha +\beta}{\mu}\frac{d}{dt}\mathbbm{E}Q(t)+\beta=-(\alpha+\beta)(1-\mathbbm{P}(Q(t)=0))+\beta.
\end{equation}
The last equation comes from the fact that when there are no arrivals the expected queue size is decreasing at rate $-\mu$ for as long as the server is busy. If $f(t_b)>0$, then from \eqref{General_cost_derivative} coupled with the equilibrium condition $c'(t)=0$ we can obtain:
\begin{equation}
(\alpha+\beta)(1-\mathbbm{P}(Q(t_b)=0))=f(t_b)\left(\frac{\alpha+\beta}{\mu}+\gamma\right)+\beta>\beta.
\end{equation}
If $F(t)$ is continuous then so is $\mathbbm{P}(Q(t)=0)$ (see for example \cite{JS2012}). So we can conclude that
\begin{equation}
c'(t_b+)=-f(t_b)\left(\frac{\alpha+\beta}{\mu}+\gamma\right)<0,
\end{equation}
which is a contradiction to the equilibrium assumption. \qed
\end{proof}
We can therefore conclude that the equilibrium solution is of the same functional form for both the tardiness cost model and the tardiness and index costs model. The key differences are that the index model has an infinite support, and the boundary condition is known (i.e. $t_a=-\frac{N\gamma}{\alpha}$). The equilibrium cost in the general case is larger than the cost in the index only case: $N\gamma$, since a customer that arrives at $t_b$ is last with probability one and there is a strictly positive probability that there are still customers in the system. We will derive a lower bound for both the individual equilibrium cost and for $-t_a$, which will be useful in the numerical computation and in the social utility analysis.
\begin{lemma}\label{General_lemma_bounds}
If we denote the individual equilibrium cost by $c_e$, then:
\begin{equation}\label{eq_ce_bound}
N\left(\frac{\beta}{\mu}+\gamma\right) < c_e < N\left(\frac{\alpha+\beta}{\mu}+\gamma\right),
\end{equation}
and consequently:
\begin{equation}\label{eq_ta_bound}
-\frac{N(\alpha+\beta+\gamma\mu)}{\alpha\mu} < t_a < -\frac{N\left(\beta+\gamma\mu\right)}{\alpha\mu}.
\end{equation}
\end{lemma}
\begin{proof}
The RHS of \eqref{eq_ce_bound} (and the LHS of \eqref{eq_ta_bound}) is trivially true for it is the cost of a customer arriving last and facing a full queue. Further the inequality is strict because assuming an equality will lead to an immediate contradiction to the equilibrium assumption. This is clearly a coarse bound, which we only use in order to bound the search range in the numerical procedure presented in the next section.

The LHS of \eqref{eq_ce_bound} is of more interest and requires a more technical proof, which follows similar steps as the proof of Lemma 10 in \cite{JS2012}, where a respective bound is established for the model with no index costs. 

We will assume that $c_e\leq N\left(\frac{\beta}{\mu}+\gamma\right)$ (or $t_a\geq -\frac{N\left(\beta+\gamma\mu\right)}{\alpha\mu}$), and show that this leads to a contradiction. We first show that this assumption leads to the conclusion that $t_b>\frac{N}{\mu}$. The characterization of the equilibrium arrival distribution in \eqref{General_f_negative_tardiness} and \eqref{General_eq_dynamics_tardiness} yields:
\begin{equation*}
\begin{split}
F\left(\frac{N}{\mu}\right) &= F(0)+\int_0^{\frac{N}{\mu}}f(t)dt \\
&= -t_a\frac{\alpha\mu}{N(\alpha+\beta+\gamma\mu)}+\int_0^{\frac{N}{\mu}}\frac{\mu(\alpha+\beta)\left(1-\mathbbm{P}(Q(t)=0)\right)-\beta\mu}{N(\alpha+\beta+\gamma\mu)}dt \\
&< \frac{\beta+\gamma\mu}{\alpha+\beta+\gamma\mu}+\int_0^{\frac{N}{\mu}}\frac{\alpha\mu}{N(\alpha+\beta+\gamma\mu)}dt=1 
\end{split}.
\end{equation*}
Hence, $F\left(\frac{N}{\mu}\right)<1$ which implies that $t_b>\frac{N}{\mu}$. Recall that the equilibrium cost needs to be constant on all of the support, and in particular $c(t_b)=c_e\leq N\left(\frac{\beta}{\mu}+\gamma\right)$. But if $t_b>\frac{N}{\mu}$, then  according to \eqref{General_cost} this is not satisfied:
\begin{equation*}
c(t_b)=\frac{\alpha+\beta}{\mu}\mathbbm{E}Q(t_b)+\beta t_b+\gamma N > \beta\frac{N}{\mu}+\gamma N \geq c_e,
\end{equation*}
and thus we have reached a contradiction.
\qed
\end{proof}

\subsection{Numerical procedure}\label{sec_general_numeric}
We can now present a numerical method to compute the equilibrium distribution in the general setting, with no closing time or early birds (both will be addressed later):
\begin{itemize}
\item[(1)] Choose an arbitrary $t_a \in [-\frac{N(\alpha+\beta+\gamma\mu)}{\alpha\mu},-\frac{N\left(\beta+\gamma\mu\right)}{\alpha\mu}]$ (Lemma \ref{General_lemma_bounds})
\item[(2)] Sequentially compute $F(t)$ for $t\geq t_a$ discretely according to the dynamics given by \eqref{General_f_negative_tardiness}, \eqref{General_eq_dynamics_tardiness}:
\begin{equation*}
f(t)= \left\{
	\begin{array}{ll}
		 \frac{\alpha\mu}{N(\alpha+\beta+\gamma\mu)} \mbox{, } &  t\in [t_a,0) \\
		\frac{\mu(\alpha+\beta)\left(1-\mathbbm{P}(Q(t)=0)\right)-\beta\mu}{N(\alpha+\beta+\gamma\mu)} \mbox{, } &  t\in[0,t_b]		
	\end{array}
\right. .
\end{equation*}
The approximation step of $\mathbbm{P}(Q(t)=0)$ from $t$ to $t+\Delta$ is given by the Kolmogorov equations in \eqref{General_Q_dynamics}:
\begin{equation*}
\begin{aligned}
	p_{0,j}(t+\Delta) &= p_{0,j}(t)+\Delta\mu p_{1,j}(t)-\Delta(N-j)h(t)p_{0,j}(t), \ 1\leq j\leq N  \\
	p_{i,j}(t+\Delta) &= p_{i,j}(t)+\Delta\mu p_{i+1,j}(t)+\Delta(N-j+1)h(t)p_{i-1,j-1}(t)    \\
				   &= -\Delta(\mu+(N-j)h(t))p_{i,j}(t), \ 1\leq i\leq j\leq N
\end{aligned},
\end{equation*}
where $\Delta>0$ is the discretization parameter. The initial conditions are given by \eqref{General_init_conditions}:
\begin{equation*}
p_{i,j}(0)={N \choose i} F(0)^i (1-F(0))^{N-i} \mathbbm{1}_{\lbrace i=j\rbrace},
\end{equation*}
and $F(0)=-t_a\frac{\alpha\mu}{N(\alpha+\beta+\gamma\mu)}$.
\item[(3)] Stop at the first $t$ such that either of these two conditions (Lemma \ref{General_lemma_tb}) is met, for some tolerance parameter $\epsilon>0$:
\begin{itemize}
\item[(a)] $|F(t)-1|\leq\epsilon$
\item[(b)] $|f(t)-0|\leq\epsilon$
\end{itemize}
\item[(4)] If both are met simultaneously, stop and set $t_b=t$.
\item[(5)] If (a), then chose a smaller $t_a$ and go back to (2).
\item[(6)] If (b), then chose a larger $t_a$ and go back to (2).
\end{itemize}
\begin{remark}[Remark 1]
The solution $f(t)$ of \eqref{General_eq_dynamics} was proven to be monotonic with respect to the initial condition $t_a$ in \cite{JS2012}. This is sufficient for the above procedure to converge when a simple bisection procedure is used.
\end{remark}
\begin{remark}[Remark 2]
If early birds are not allowed the procedure can be adjusted by running the search procedure on $p_0$. The discrete computation will now start at $t_e$ (computed using equation \eqref{General_te_equation}) with initial conditions:
\begin{equation}
p_{i,j}(t_e)={N \choose j} p_0^j (1-p_0)^{N-j} \frac{e^{-\mu t_e}(\mu t_e)^{j-i}}{(j-i)!} \mathbbm{1}_{\lbrace j\geq i\rbrace}.
\end{equation}
\end{remark}
\begin{remark}[Remark 3]
An additional adjustment can be made if we assume closing time $T<\infty$, such that no arrivals are admitted into the queue after this time. If $t_b\leq T$, then no adjustment is required. Otherwise, the conditions in (3) can be replaced by $F(T)=1$.
\end{remark}
\subsubsection{Computational advantage of the index model}\label{sec_general_complex}
Suppose that there is no tardiness cost and no closing time, then $F(t)<1$ for any $t<\infty$. This can be addressed by adding a tolerance parameter $\epsilon>0$, and adjusting the stopping rule in (3) to $F(t)=1-\epsilon$. We denote this time by $t_\epsilon:=\{t:F(t)=1-\epsilon\}$. We next derive the computational complexity for this case and explain why it is lower than in the general model.
\begin{lemma}\label{Lemma_complexity}
If $\beta=0$ and $T=\infty$, then the computational complexity of the procedure for $N$ customers, with parameters $\Delta$ and $\epsilon$ is $\mathcal{O}\left(\frac{-\log(\epsilon)N^2}{2\Delta}\right)$.
\end{lemma}
\begin{proof}
If $\beta=0$ and $T=\infty$, then the initial conditions are known: if early birds are allowed then $t_a=-\frac{N\gamma}{\alpha}$, and if they are not then:
\begin{equation*}
p_0 =
\left\{
	\begin{array}{ll}
		 \frac{\gamma}{\gamma+\frac{\alpha}{\mu}} & \mbox{, } \frac{\alpha}{\mu}>\gamma  \\
		1 & \mbox{, }  \frac{\alpha}{\mu}\leq\gamma
	\end{array}
\right. .
\end{equation*}
Hence, no bisection is needed and there is only one iteration of the algorithm. 

We now compute the number of approximation steps required in one iteration. The number of partitions the interval $[0,t_\epsilon]$ in the discretization is $\frac{t_\epsilon+1}{\Delta}$. In each step the probability of all states of the process $\{Q(t),A(t)\}$ is approximated. The size of the state space is 
\begin{equation}
\left|\{(i,j):0\leq i\leq N, \ i\leq j\leq N\}\right|=\frac{(N+1)(N+2)}{2}.
\end{equation}
We can conclude that the total number of approximations in the procedure is $\frac{(t_\epsilon+1)(N+1)(N+2)}{2\Delta}$. Finally, in Lemma \ref{General_lemma_exponential_tail} we showed that the $F$ has an exponentially decreasing tail, with a rate of $0<\eta<\mu$. Thus for small $\epsilon$ we can approximate $t_\epsilon$ by solving the equation:
\begin{equation}
1-F(t_\epsilon)=e^{-\eta t},
\end{equation}
yielding $t_\epsilon\sim -\log(\epsilon)$.
\qed
\end{proof}

If we assume tardiness costs $\beta>0$ or a closing time $T<\infty$ (or both), then the initial conditions are not known and there is a need for a bisection search for $t_a$ (or equivalently, for $p_0$). The complexity of each iteration in the search is given by Lemma \ref{Lemma_complexity}: $\mathcal{O}\left(\frac{t_b N^2}{2\Delta}\right)$ if $\beta>0$ and $\mathcal{O}\left(\frac{T N^2}{2\Delta}\right)$ if $\beta=0$. Note that the number of required iterations for the bisection to converge is a function the tolerance parameter $\epsilon$, and so is the approximated value of $t_b$ in case $\beta>0$. In both cases the computational resources required in each iteration is of similar magnitude as in the index only case, but multiple (possibly many) iterations are required. 

\subsection{Numerical examples}\label{sec_general_examples}
We applied the numerical procedure presented above in order to compute the equilibrium arrival distribution for several examples. First of all for the case with no tardiness cost ($\beta=0$), in Figure \ref{fig:figure_general_noBeta} we see four arrival densities, for different population sizes. In all four examples the service rate is $\mu=20$ and the waiting cost is $\alpha=0.1$. By keeping the product $N\gamma$ constant at a value of one, in all of the examples, we also keep the individual equilibrium cost constant (and equal to one). This also implies that the lower bound of the support remains constant at $t_a=-\frac{N\gamma}{\alpha}=10$. The only change is the shape of the distribution which becomes more "spread out" as the population size increases. It is evident that for a larger number of customers the decrease in the arrival rate is not exponential at first, but rather starts moderately and is only exponential in the tail. 
\begin{figure}[H]
\begin{center}
\includegraphics[width=13cm,height=4.7cm	]{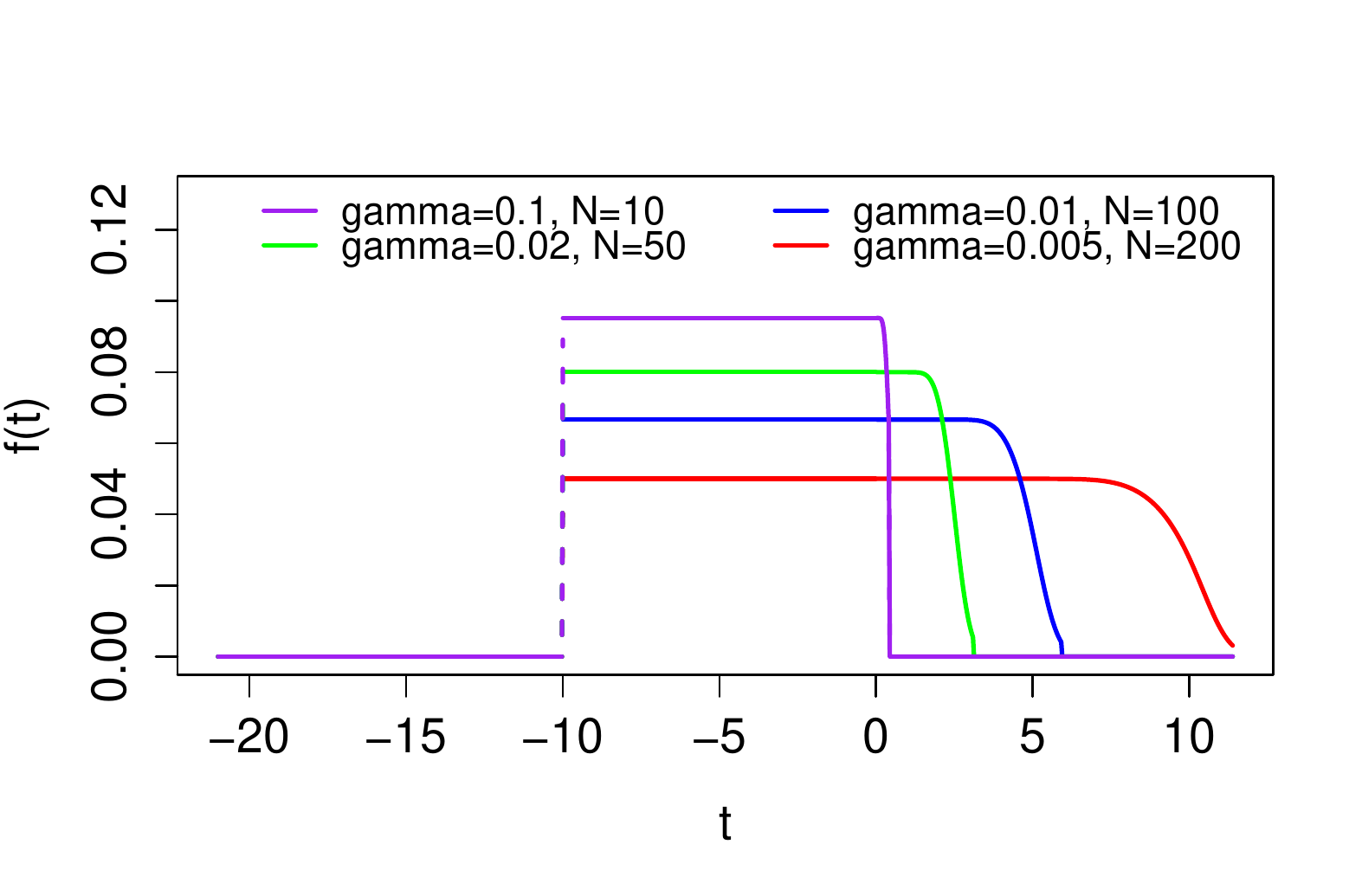}
\end{center}
\caption{Equilibrium arrival density \ $\left(\mu=20, \ \alpha=0.1, \ \beta=0 \right)$}
\label{fig:figure_general_noBeta}
\end{figure}

In figure \ref{fig:figure_general_tail} we show the tail behaviour of the hazard rate for increasing population sizes when all other parameters are kept constant. For $N>1$ the hazard rate is not constant, but as we have shown in Lemma \ref{General_lemma_exponential_tail} it approaches a constant rate. The numerical examples suggest that this constant rate is the same as in the case of $N=1$, which equals $\mu F(0)=\frac{\mu}{1+\frac{\alpha}{\gamma\mu}}$. In other words, the tail behaviour of the arrival distribution is dependent on the cost parameters, but no on the population size.
\begin{figure}[H]
\begin{center}
\includegraphics[width=13cm,height=4.7cm	]{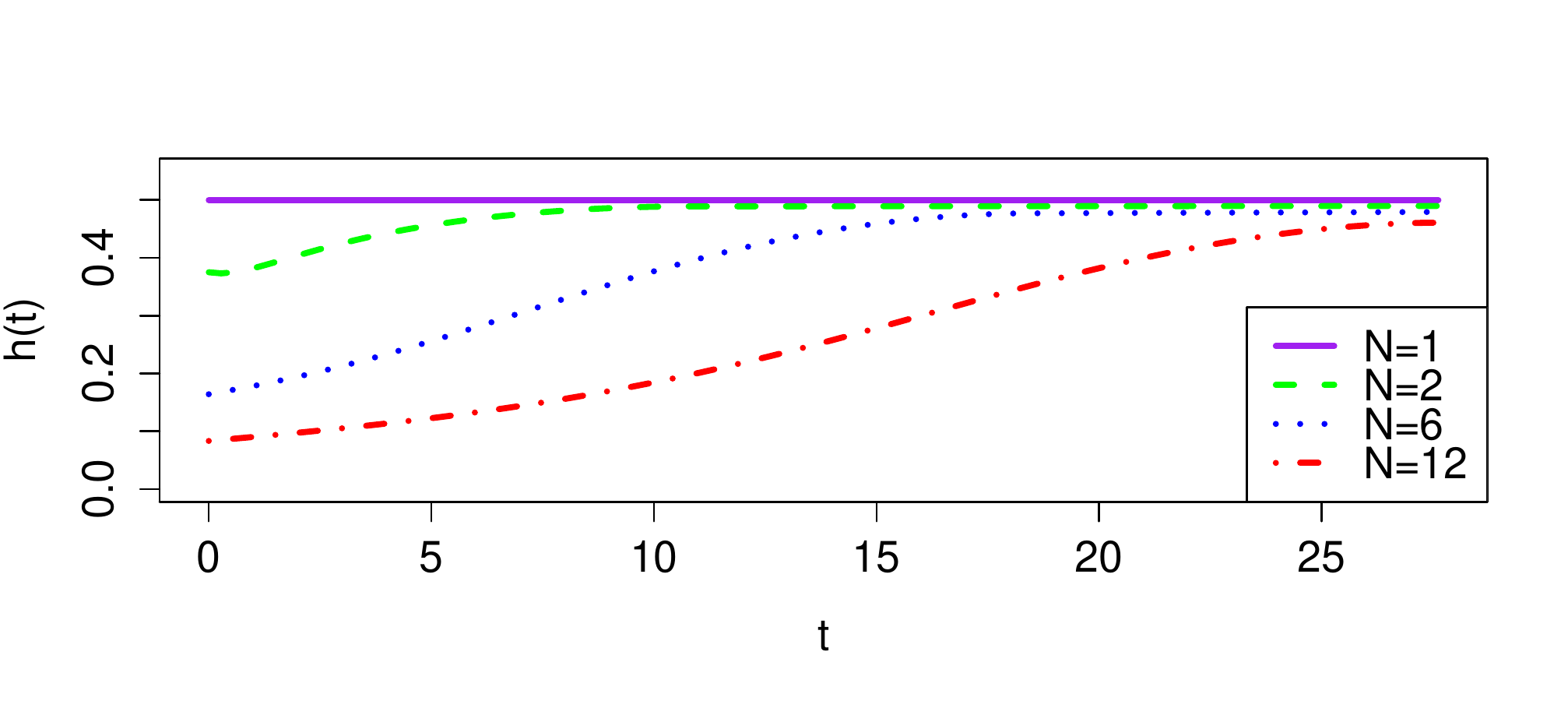}
\end{center}
\caption{Hazard rate for different population sizes \ $\left(\mu=1, \ \alpha=1, \ \gamma=1 \right)$}
\label{fig:figure_general_tail}
\end{figure}

The next examples were computed for the same sets of parameters as in figure \ref{fig:figure_general_noBeta}, with an addition of a tardiness cost: $\beta=0.1$. These computations are shown in Figure \ref{fig:figure_general_Beta}. In this case, the support is finite: $[t_a,t_b]$. Increasing the population now makes the support "spread out" in both directions, with customers starting to arrive earlier, and with a more moderate decrease after time zero.
\begin{figure}[H]
\begin{center}
\includegraphics[width=13cm,height=4.7cm	]{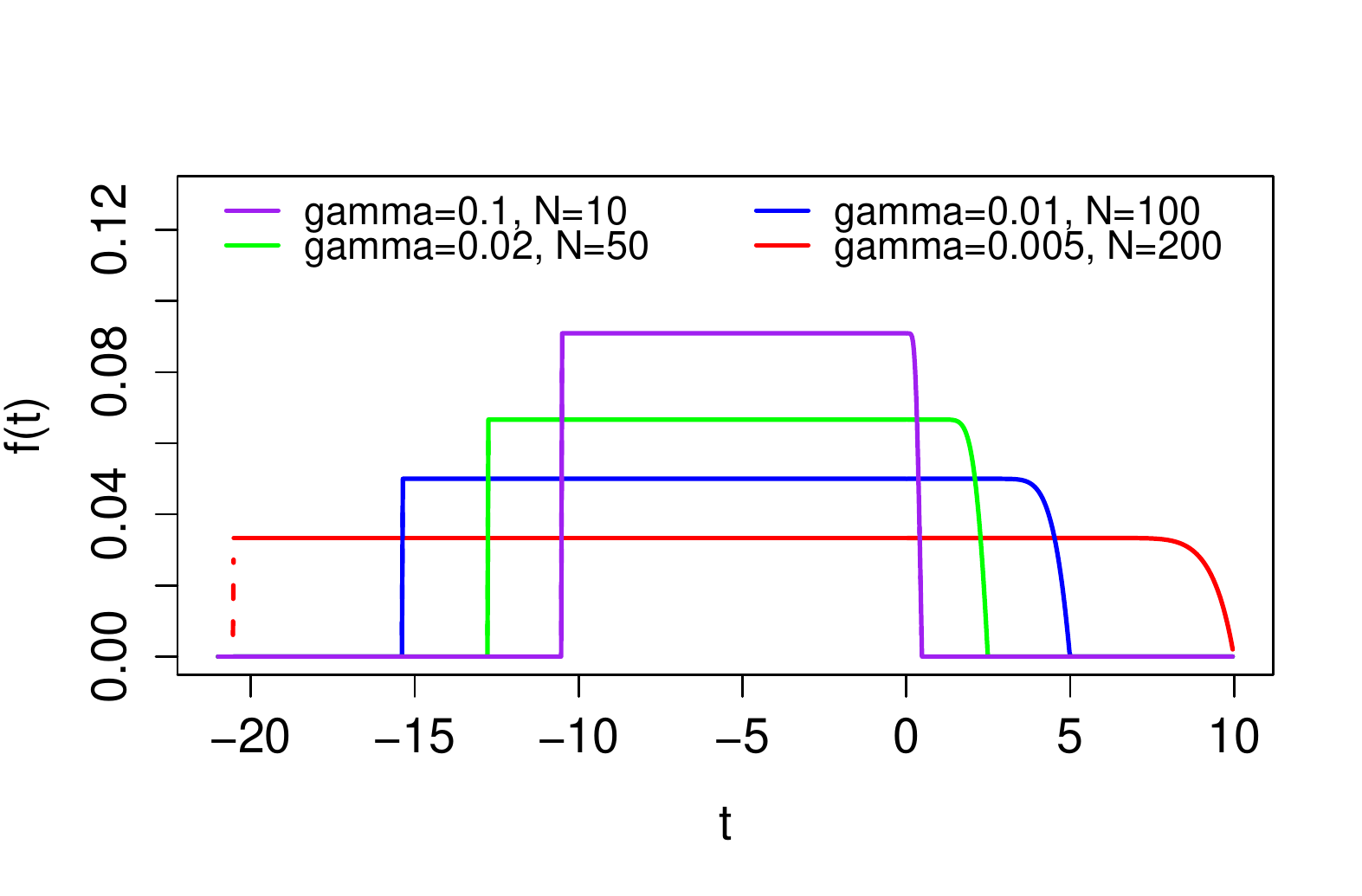}
\end{center}
\caption{Equilibrium arrival density \ $\left(\mu=20, \ \alpha=0.1, \ \beta=0.1 \right)$}
\label{fig:figure_general_Beta}
\end{figure}

\section{Poisson number of customers}\label{sec_poisson}
A common model assumption in queueing arrival games is that the number of customers arriving to the system follows a Poisson distribution with some parameter $\lambda$. Although this assumption might seem to make the model more complicated, it actually simplifies quite a few of the calculations. Namely, if each customer believes that the number of other customers is $N\sim Poisson(\lambda)$ \footnote{It turns out this is consistent with assuming the overall number of customers $N+1$ is of the same distribution. Moreover, this phenomenom only holds for Poisson random variables. See for example McAfee and McMillan \cite{MM1987} or Haviv and Milchtaich \cite{HM2012}.} and that they all arrive according to the same distribution $F$, then the arrival process $A(t)$ is a non-homogeneous in time Poisson process with rate $\lambda f(t)$. The independent increments then imply that the number of customers yet to arrive at any given time does not depend on the number of customers that have already arrived. We assume a linear cost function with $\alpha>0$ and $\beta,\gamma\geq 0$, no closing time and that early birds are allowed. Analysis very similar to section \ref{sec_general}, leads to the following equilibrium arrival distribution:

\begin{equation} \label{Poiss_equil_density}
f(t) =
\left\{
	\begin{array}{ll}
		\frac{\alpha\mu}{\lambda(\alpha+\beta+\gamma\mu)} \  & \mbox{, } \ t\in [t_a,0) \\
		\frac{\mu(\alpha+\beta)\left(1-p_0(t)\right)-\beta\mu}{\lambda(\alpha+\beta+\gamma\mu)} & \mbox{, } \ t\in [0,t_b] \\
		0  & \mbox{, } o.w.
	\end{array}
\right. ,
\end{equation}
such that $\int_{t_a}^{t_b} f(t)dt=1$ , $f(t_b)=0$, and $p_i(t)=\mathbbm{P}(Q(t)=i)$ satisfy:
\begin{equation} \label{Poisson_queue_dynamics}
p_{i}^{'}(t)=
\left\{
	\begin{array}{ll}
		\mu p_{1}(t))-\lambda f(t)p_{0}(t) \ & \mbox{, } i=0 \\
		\mu p_{i+1}(t)+\lambda f(t)p_{i-1}(t)-(\mu+\lambda f(t))p_{i}(t) \ & \mbox{, }  i>0
	\end{array}
\right. .
\end{equation}
Unsurprisingly, the equilibrium solution is of the same functional form as in \cite{GH1983} and \cite{H2013}. If $\beta=0$, the arguments of Lemma \ref{Lemma_index_no_tb} still hold. To be specific, any customer can ensure an expected cost of no more than the expected number of customers times the order penalty. Therefore, the equilibrium cost is $\lambda\gamma$ and the arrival support is $[-\frac{\lambda\gamma}{\alpha},\infty)$. This means that the computational advantages discussed in the precious section are still valid in the Poisson case.

\section{Social optimization and Price of Anarchy}\label{sec_social}
When considering the goal of minimizing the total cost inured by the customers, thus maximizing social welfare, then the inclusion of index costs has no impact on the optimization problem: The order of admittances does not matter because there is always a first customer, last customer and so on. If there are $N+1$ customers, then the total cost of the order penalties is $\frac{\gamma N(N+1)}{2}$, regardless of the actual order. Including this cost in the definition of the price of anarchy yields:
\begin{equation}\label{eq_PoA}
PoA:=\frac{(N+1)c_e}{\frac{\gamma N(N+1)}{2}+c_{opt}},
\end{equation}
where $c_{opt}$ is the total socially optimal expected waiting and tardiness cost. We will review two examples of social optimization for this model, the first a scheduling policy when there are no tardiness costs and the second a dynamic policy for the general model.

\subsection{No tardiness costs}\label{sec_social_notardiness}
If there are no tardiness costs, then the socially optimal schedule is one that minimizes the expected queue size at arrival times. This problem is only relevant when there is a closing time $T$, for otherwise a central planner could "spread out" the customers as much he wants and achieve a cost of zero. 
\subsubsection{Example: $N+1=3$}
If $N=3$ and $T=1$ it can be verified that scheduling one arrival at zero, another at $T$ and a third at
\begin{equation}
t^{*} =
\left\{
	\begin{array}{ll}
		\frac{1}{\mu}\log\left(\frac{1+\sqrt{1+4e^{\mu}}}{2}\right)  & \mbox{, } \frac{1}{\mu}\log\left(\frac{1+\sqrt{1+4e^{\mu}}}{2}\right) \leq 1 \\
		1 & \mbox{, } o.w.
	\end{array}
\right.,
\end{equation}
is the socially optimal schedule. And the total socially optimal waiting cost is:
\begin{equation} \label{social_example}
c_{opt}=\alpha\left(e^{-\mu t^{*}}+ e^{-\mu}(1+\mu-\mu t^{*}+e^{\mu t^{*}})\right)+3\gamma.
\end{equation}
Table \ref{tbl_poa_N3} shows exact computations of the PoA for several model parameters. When the index cost parameter is large, relative to the waiting cost parameter then the price of anarchy is lower, as expected since the central planner has less impact on the total cost. Moreover, the $PoA$ can be very close to one when the index parameter is the dominant factor in the cost function.

\begin{table}[H]
\centering
\caption{Price anarchy computations when $N+1=3$ and $\alpha=1$.}
\begin{tabular}{|c|c|c|c|} \hline 
$PoA$ & $\gamma=0.05$ & $\gamma=1$ & $\gamma=5$  \\ \hline
$\mu=0.5$ & 4.025 & 2.909 & 2.273 \\ \hline
$\mu=1$ & 2.079 & 2.002 & 1.995 \\ \hline
$\mu=2$ & 1.266 & 1.731 & 1.928\\ \hline 
\end{tabular}
\label{tbl_poa_N3}
\end{table}

\subsubsection{General number of customers}
For a general number of customers, finding an optimal schedule for this problem is a hard global optimization problem, and typically can only be solved using heuristic or approximation algorithms. Examples of such algorithms can be found in Pedgen and Rosenshine \cite{PR1990} or Stein et al. \cite{SCM1994}. More recently Hassin and Mendel \cite{HM2008} presented an optimization procedure in the context of patient scheduling, when some proportion of the customers do not show up. An alternative formulation of the social optimization problem is suggested in \cite{HK2010}: The central planner may not specify individual arrival times to the customers, but rather a distribution for all customers to randomly draw their arrivals from. They provide an approximation algorithm for finding the socially optimal distribution which indicates that the optimal solution is approximately a uniform distribution.

\subsection{Tardiness and index costs}\label{sec_social_general}
A strict lower bound of 2 for the PoA was established in \cite{JS2012}, along with convergence to the bound when the number of customers increases. They do this by comparing the equilibrium cost to an optimal schedule where the central planner observes the system and can send a new customer each time a service is completed. This dynamic policy is still optimal when index costs are included and then by definition \eqref{eq_PoA} we get:
\begin{equation}\label{PoA_no_tardiness}
PoA=\frac{(N+1)c_e}{\frac{\gamma N(N+1)}{2}+\frac{\beta N(N+1)}{2\mu}}.
\end{equation}
In Lemma \ref{General_lemma_bounds} we showed that $c_e>N\left(\frac{\beta}{\mu}+\gamma\right)$. This leads directly to the following lemma:
\begin{lemma}
In the general cost model, if the central planner can schedule arrivals dynamically, then $PoA>2$. 
\end{lemma}
We can conclude that even though the central planner cannot reduce the total index costs, the price of anarchy is still always greater than two, as in the model without them. This is another property that holds when generalizing the standard model to include index costs.

\section{Concluding Remarks}\label{sec_conclusion}
We have introduced a new cost model for the queueing arrival game, which penalizes late arrivals for their order of admittance, instead (or in addition) to the previously modelled tardiness costs. We have shown that there are several differences between these two models. In particular, if just order costs are assumed, than the arrival rate decreases exponentially, but does not reach zero at any finite time. Furthermore, we have shown that the equilibrium cost, in this case, depends solely on the order penalty parameter. We have also shown how the equilibrium behaviour changes when limiting the customer arrival interval, by not allowing to queue before opening time or by setting a closing time. The equilibrium arrival process characterized here is qualitatively similar to the standard tardiness model (\cite{JS2012} and \cite{H2013}), i.e. a uniform arrival distribution before the server opens and a decreasing arrival density (with the same functional form) afterwards. However, the index model is easier to numerically approximate because the initial conditions of the distribution are derived directly from the model assumptions. The role of the tardiness cost assumption has been motivated in the literature as an indirect proxy for an index cost, for example in the concert example. We have shown here how to take this cost into account directly. \\

There are a few avenues for further research on this model: considering non-linear cost functions, populations with non-homogeneous cost parameters or a general service distribution. In the case of non-linear waiting and order costs, necessary equilibrium conditions may be developed using the techniques presented here. However, it is hard to provide general sufficient conditions for the uniqueness or even existence of a Nash equilibrium. This analysis may be carried out for special cases, such as a quadratic cost function. It is reasonable to assume that customers differ in their utility functions, and so the analysis of non-homogeneous customer types is much called for. This has been considered lately in a queueing context by Guo and Hassin in \cite{GH2012}, and in the closely related queueing arrival game in \cite{JJS2011}. The latter studies the equilibrium arrival pattern in the fluid setting, and extending this analysis to discrete customer models like the one presented in this paper suggests an interesting research challenge. An additional question that stems directly from our research is that of characterizing the quasi-stationary distribution of the queueing process, as defined in \cite{vDP2013}.

\section*{Acknowledgments}
The author would like to thank Moshe Haviv, Binyamin Oz, Refael Hassin and two anonymous referees for their valuable comments and advice throughout this work. The author gratefully acknowledges the financial support of the Israel Science Foundation grant no. 1319/11 and the Center for the Study of Rationality in the Hebrew University of Jerusalem.
{\footnotesize\bibliography{WhenToBib}}

\end{document}